\documentclass[12pt]{article}
\usepackage{a4}
\usepackage{amsthm}
\usepackage{amsfonts}
\usepackage{amssymb}
\usepackage{amsmath} 
\usepackage{array}
\usepackage{graphicx}
\usepackage{cite}
\newtheorem{theorem}{Theorem}
\newtheorem{lemma}[theorem]{Lemma}
\newtheorem{conjecture}{Conjecture}
\newtheorem{proposition}[theorem]{Proposition}
\def\Aut{{\rm Aut}}
\newcommand{\name}{hypercubical\ }
\newcommand{\recipe}{recipe}
\newcommand{\checkered}{diagonal checker}
\newcommand{\graphon}{W_{\boxplus}}
\newcommand{\constr}{\mathcal{C}_{\boxplus}}
\newcommand{\cube}{r}
\newcommand{\ccube}{\mathfrak{R}}
\newcommand{\twoblock}[1]{\langle#1\rangle}
\newcommand{\A}{A^{\boxplus}}
\newcommand{\B}{B^{\boxplus}}
\newcommand{\C}{C^{\boxplus}}
\newcommand{\D}{D^{\boxplus}}
\newcommand{\E}{E^{\boxplus}}
\newcommand{\F}{F^{\boxplus}}
\newcommand{\X}{X^{\boxplus}}
\newcommand{\Y}{Y^{\boxplus}}
\newcommand{\JJ}{\mathcal{J}}
\newcommand{\WW}{W}
\newcommand{\NN}{{\mathbb N}}
\newcommand{\nati}{\NN^{\ast}}
\newcommand{\re}[1]{\widehat{#1}}

\def\dif{{\rm d}}

\begin{document}
\title{Infinite dimensional finitely forcible graphon\thanks{The~work leading to this invention has received funding from the European Research Council under the European Union's Seventh Framework Programme (FP7/2007-2013)/ERC grant agreement no.~259385. A part of this work was done during a visit of the last two authors to the Institut Mittag-Leffler (Djursholm, Sweden).}}
\author{Roman Glebov\thanks{School of Computer Science and Engineering, Hebrew University, Jerusalem 9190401, Israel.
			    E-mail: {\tt roman.l.glebov@gmail.com}.
			    Previous affiliations: Mathematics Institute and DIMAP, University of Warwick, Coventry CV4 7AL, UK.
                            Department of Mathematics, ETH, 8092 Zurich, Switzerland.
			    }\and
	Tereza Klimo{\v s}ov\'a\thanks{Department of Applied Mathematics, Faculty of Mathematics and Physics, Charles University, Malostransk\'e n\'am.~25, 118 00 Prague 1, Czech Republic. E-mail: {\tt tereza@kam.mff.cuni.cz}. Previous affiliation: Mathematics Institute and DIMAP, University of Warwick, Coventry CV4 7AL, UK. This author was supported by Center of Excellence --- ITI, project P202/12/G061 of GA \v{C}R and by the Center for Foundations of Modern Computer Science (Charles University project UNCE/SCI/004).}
\and
        Daniel Kr\'al'\thanks{Faculty of Informatics, Masaryk University, Botanick\'a 68A, 602 00 Brno, Czech Republic, and Mathematics Institute, DIMAP and Department of Computer Science, University of Warwick, Coventry CV4 7AL, UK. E-mail: {\tt dkral@fi.muni.cz}. This author was also supported by the European Research Council (ERC) under the European Union’s Horizon 2020 research and innovation programme (grant agreement No 648509). This publication reflects only its authors' view; the European Research Council Executive Agency is not responsible for any use that may be made of the information it contains.}
	}

\date{}
\maketitle
\begin{abstract}
Graphons are analytic objects associated with convergent sequences of dense graphs. 
Finitely forcible graphons, i.e., those determined by finitely many subgraph densities,
are of particular interest because of their relation to various problems in extremal combinatorics and theoretical computer science.
Lov\'asz and Szegedy conjectured that the topological space of typical vertices of a finitely forcible graphon
always has finite dimension,
which would have implications on the minimum number of parts in its weak $\varepsilon$-regular partition.
We disprove the conjecture by constructing a finitely forcible graphon with the space of typical vertices
that has infinite dimension.
\end{abstract}

\section{Introduction}
\label{sec:intro}

Analytic objects associated with convergent sequences of combinatorial objects
have recently attracted significant amount of attention. This line of
research was initiated by the theory of limits of dense graphs~\cite{bib-borgs08+,bib-borgs+,bib-borgs06+,bib-lovasz06+},
followed by limits of sparse graphs~\cite{bib-bollobas11+,bib-elek07},
permutations~\cite{bib-hoppen-lim1,bib-hoppen-lim2}, partial orders~\cite{bib-janson11} and others.
Analytic methods applied to such limit objects led to results
in many areas of mathematics and computer science, in particular in extremal combinatorics~\cite{bib-flag1, bib-flag2, bib-flagrecent, bib-flag3, bib-flag4, bib-flag5, bib-flag6, bib-flag6.5, bib-flag7, bib-flag8, bib-flag9, bib-flag10, bib-razborov07,bib-flag11, bib-flag12} and property testing~\cite{bib-hoppen-test,bib-lovasz10+}.

In this paper we are concerned with limits of dense graphs and
in particular with those determined by finitely many subgraph densities.
This phenomenon, which is known as finite forcibility, is closely related to quasirandomness of combinatorial objects,
whose study was initiated by Chung, Graham and Wilson~\cite{bib-chung89+},
R\"odl~\cite{bib-rodl} and Thomason~\cite{bib-thomason, bib-thomason2}.
In the setting of graph limits,
large dense graphs are represented by analytic objects called graphons and
the just mentioned results assert that every constant graphon is finitely forcible.
This result was generalized by Lov\'asz and S\'os~\cite{bib-lovasz08+}, see also~\cite{bib-spencer10},
who proved that every step graphon,
which is a multipartite graphon with uniform densities between and within its parts, is finitely forcible.

We are interested in the structure of the space of typical vertices of finitely forcible graph limits.
We consider two spaces of typical vertices, which we formally define in Section~\ref{sec:def}.
One is the space studied in~\cite{bib-lovasz11+} and is denoted by $T(W)$;
informally speaking,
$T(W)$ is formed by the neighbor functions (``rows'' of a graphon $W$) with the $L^1$-topology.
The other, which is denoted by $\overline{T}(W)$, is the space studied in~\cite[Chapter 13]{bib-lovasz-book},
where the $L^1$-metric is replaced by a finer metric.
The structure of the space $\overline{T}(W)$
is closely related to weak $\varepsilon$-regular partitions of $W$~\cite{bib-lovasz-book,bib-lovasz07+};
in particular, if $\overline{T}(W)$ has finite Minkowski dimension,
then $W$ has a weak $\varepsilon$-regular partition with a number of parts polynomial in $\varepsilon^{-1}$. 
We note that there are graphons $W$ such that
the minimum number of parts in a weak $\varepsilon$-regular partition of $W$
is exponential in $\varepsilon^{-2}$~\cite{bib-conlon12+}.
In particular, graphons $W$ such that the Minkowski dimension of $\overline{T}(W)$ is finite
have simple structure from the regularity decomposition point of view.

Lov\'asz and Szegedy~\cite[Conjecture 10]{bib-lovasz11+},
led by examples of finitely forcible graphons that were known at that time,
conjectured that the space of typical vertices of a finitely forcible graphon always has finite dimension.
We cite the conjecture verbatim.

\begin{conjecture}
\label{conj:2}
If $W$ is a finitely forcible graphon, then $T(W)$ is finite dimensional.
(We intentionally do not specify which notion of dimension is meant here---a result concerning any variant would be interesting.)
\end{conjecture}

\noindent In this paper we construct a graphon $\graphon$, which we call a {\em \name} graphon,
such that $\graphon$ is finitely forcible and
both $T(\graphon)$ and $\overline{T}(\graphon)$ contain subspaces homeomorphic to $[0,1]^{\infty}$.

\begin{theorem}\label{thm:main}
The \name graphon $\graphon$ is finitely forcible and
the topological spaces $T(\graphon)$ and $\overline{T}(\graphon)$
contain subspaces homeomorphic to $[0,1]^{\infty}$ equipped with the product topology.
\end{theorem}

\noindent Looking at one of the motivations for studying the dimension of the spaces $T(W)$ and $\overline{T}(W)$,
we remark that every weak $\varepsilon$-regular partition of $\graphon$ has at least $2^{\Theta(\log^2\varepsilon^{-1})}$ parts.
We further discuss
the existence of finitely forcible graphons with no weak $\varepsilon$-regular partition with a small number of parts
in the concluding section.

The proof of Theorem~\ref{thm:main} extends the methods from~\cite{bib-rademacher} and~\cite{bib-norine-comm}.
In particular, Norine~\cite{bib-norine-comm} constructed finitely forcible graphons
with the space of typical vertices of arbitrarily large (but finite) Lebesgue dimension.
In his construction, both $T(W)$ and $\overline{T}(W)$ contain a subspace homeomorphic to $[0,1]^{d}$.
One of the contributions of this paper is showing how the techniques from~\cite{bib-rademacher} and~\cite{bib-norine-comm}
can be refined to force a subspace homeomorphic to $[0,1]^{\infty}$,
which turned out to be quite challenging.
Another contribution of the paper is formalizing the methods used in~\cite{bib-rademacher} and~\cite{bib-norine-comm},
which are further used in the follow up papers~\cite{bib-fup-CKKN,bib-fup-CKM,bib-fup-GKL}.

We finish with giving a brief outline of the proof of Theorem~\ref{thm:main} in informal terms.
As in~\cite{bib-rademacher},
the constructed \name graphon $\graphon$ has several parts (see Figure~\ref{fig:graphon}),
which are determined by the degrees of the vertices that are contained in the parts.
The parts $A_1,\ldots,A_3$ serve to further partition the parts $B_1,\ldots,B_5$ into infinitely many smaller parts;
the part $B_1$ is split into parts $B_{1,d}$, $d\in\NN$.
The structure between the parts $A_1$ and $A_0$ plays the role of identifying the first of the smaller parts and
the structure between $A_1$ and $A_3$ links consecutive smaller parts.
The part $C$ serves to introduce coordinate systems on the parts $A_0,\ldots,A_3$ and $B_1,\ldots,B_5$.
The structure between the parts $B_1$ and $B_2$ provides a $d$-dimensional coordinate system on $B_{1,d}$, $d\in\NN$, and
is used to arrange that $B_{1,d}$ induces a subspace homeomorphic to $[0,1]^d$.
The $d$-dimensional structure of the parts $B_{1,d}$ is forced in an iterative (induction like) way,
increasing the dimension by one at each step.
The proof is concluded by forcing the parts $B_{1,d}$ to be ``projections'' of the part $D$;
in this way, we arrange that the subspace associated with the part $D$ is homeomorphic to $[0,1]^{\infty}$.

\section{Definitions}
\label{sec:def}

In this section we present the notation that we use throughout the paper;
this includes the notions from the theory of graph limits,
which originated in~\cite{bib-borgs06+,bib-borgs08+,bib-borgs+,bib-lovasz06+}.

A {\em graph} is a pair $(V,E)$ where $E\subseteq \binom{V}{2}$.
The elements of $V$ are called {\em vertices} and
the elements of $E$ are called {\em edges}.
All graphs considered in this paper are simple, i.e., without loops and parallel edges.
The {\em order} of a graph $G$ is the number of its vertices and
is denoted by $|G|$.
We use $\nati$ for $\NN\cup \{\infty\}$ and $[k]$ for $\{1,\ldots, k\}$.

The {\em density} of a graph $H$ in a graph $G$, which is denoted by $d(H,G)$,
is the probability that
a random set of $|H|$ distinct vertices of $G$ induce a subgraph isomorphic to $H$.
If $|H|>|G|$, we define $d(H,G)$ to be zero.
A sequence of graphs $(G_i)_{i\in\NN}$ is {\em convergent}
if the sequence $(d(H,G_i))_{i\in\NN}$ converges for every graph $H$.
In general, we will consider sequences of graphs with their orders tending to infinity.

Convergent sequences of graphs can be associated with an analytic limit object,
which we will now introduce.
A {\em graphon\/} $W$ is a symmetric measurable function from $[0,1]^2$ to $[0,1]$.
Here, {\em symmetric\/} stands for the property that $W(x,y)=W(y,x)$ for every $x,y\in [0,1]$.
Very imprecisely speaking,
one can think of a graphon as of a continuous version of the adjacency matrix of a graph.
Mimicking the terminology for graphs,
we refer to a graphon $W$ restricted to $S\times T$, where $S$ and $T$ are two measurable subsets of $[0,1]$,
as to a {\em subgraphon\/} of $W$ induced by $S\times T$.

We next link graphons to convergent sequences of graphs.
A {\em $W$-random graph\/} of order $k$ is
obtained by sampling uniformly and independently $k$ random points $x_1,\ldots, x_k\in [0,1]$,
which are associated with the vertices, and
by joining the vertices corresponding to $x_i$ and $x_j$ by an edge with probability $W(x_i,x_j)$.
Because of this connection, we refer to the points of $[0,1]$ as to the {\em vertices\/} of $W$.
The {\em density} of a graph $H$ in a graphon $W$ is the probability that
the $W$-random graph of order $|H|$ is isomorphic to $H$.
The definition of a $W$-random graph yields the following:
$$d(H,W)=\frac{|H|!}{|\Aut(H)|}\int_{{[0,1]}^{|H|}} \prod_{(i,j)\in E(H)} W(x_i,x_j) \prod_{(i,j)\not\in E(H)} (1-W(x_i,x_j)) \,\;\dif\lambda_{|H|}\;\mbox{,}$$
where $\Aut(H)$ is the automorphism group of $H$.
Our results do not depend on whether we work with Borel or Lebesgue measure on $[0,1]^d$, and
we have made a choice of working with the Borel measure throughout the paper,
which is denoted by $\lambda$ or by $\lambda_d$ if we wish to emphasize the dimension of the support space.
When we talk about the measure on $[0,1]^\NN$, we mean the product measure of the measures on $[0,1]$.

One of the key results in the theory of graph limits asserts~\cite{bib-lovasz06+} that
for every convergent sequence $(G_i)_{i\in\NN}$ of graphs with increasing orders,
there exists a graphon $W$, which is called the {\em limit} of the sequence, such that for every graph $H$,
\[d(H,W)=\lim_{i\to\infty} d(H,G_i)\;\mbox{.}\]
Conversely, if $W$ is a graphon, then the sequence of $W$-random graphs with increasing orders
converges with probability one and its limit is $W$.

Two graphons $W_1$ and $W_2$ are {\em weakly isomorphic} if $d(H,W_1)=d(H,W_2)$ for every graph $H$.
If $\varphi:[0,1]\to[0,1]$ is a measure preserving map,
then the graphon $W^\varphi(x,y):=W(\varphi(x),\varphi(y))$ is always weakly isomorphic to $W$.
The opposite is true in the following sense~\cite{bib-borgs10+}:
if two graphons $W_1$ and $W_2$ are weakly isomorphic,
then there exist measure preserving maps $\varphi_1:[0,1]\to [0,1]$ and $\varphi_2:[0,1]\to [0,1]$
such that $W_1^{\varphi_1}=W_2^{\varphi_2}$ almost everywhere.

The {\em degree} $\deg^{W} x$ of a vertex $x\in [0,1]$ in a graphon $W$ is defined as
$$\deg^{W} x=\int_{[0,1]}W(x,y)\dif y\;\mbox{.}$$
Note that the degree is well-defined for almost every vertex of $W$.
We omit the superscript $W$ whenever the graphon is clear from context. 
Let $A$ be a measurable non-null subset of $[0,1]$.
The {\em relative degree} $\deg^{W}_{A} x$ of a vertex $x\in [0,1]$ with respect $A$ is defined as
$$\deg^{W}_{A} x=\frac{\int_{A}W(x,y)\dif y}{\lambda(A)}\;\mbox{.}$$

Fix a graphon $W$, $x,x'\in [0,1]$ and a measurable set $Y\subseteq [0,1]$.
The set $N_{Y}(x)$ is the set of $y\in Y$ such that $W(x,y)>0$ and 
$$N_{Y}(x\setminus x')=\{y\in Y\ |\ W(x,y)>0\mbox{ and }W(x',y)<1\}.$$
Informally speaking, $N_Y(x\setminus x')$ contains $y\in Y$ such that
a vertex associated with $y$ can be a neighbor of a vertex associated with $x$ and
a non-neighbor of a vertex associated with $x'$ in a $W$-random graph.
We note that, assuming that $Y$ is measurable,
$N_{Y}(x)$ and $N_Y(x\setminus x')$ are measurable
for almost every $x$ and almost every pair $x$ and $x'$, respectively.

As mentioned in the Introduction, the structure and the complexity of a graphon can be studied
by analyzing a topological space associated with its typical vertices~\cite{bib-lovasz10++}.
We now give the formal definitions of the two types of such spaces that we mentioned in the Introduction.
For a graphon $W$ and $x\in [0,1]$, define a function $f^W_x:[0,1]\to [0,1]$ to be
$$f^W_x(y):=W(x,y).$$
Since the function $f^W_x$ belongs to $L^1([0,1])$ for almost every $x\in [0,1]$,
the graphon $W$ naturally defines a probability measure $\mu$ on $L^1([0,1])$.
The space $T(W)$ is formed by the support of the measure $\mu$ equipped with the topology inherited from $L^1([0,1])$.
A vertex $x$ of the graphon $W$ is called {\em typical\/} if $f^W_x\in T(W)$.
Another topological space, which is denoted by $\overline{T}(W)$,
can be defined using the notion of {\em similarity distance\/}.
If $f$ and $g$ are two functions from $L^1([0,1])$, define
\[d_W(f,g):=\int\limits_{[0,1]}\left|\;\int\limits_{[0,1]} W(x,y)(f(y)-g(y))\dif y\right| \dif x\;\mbox{.}\]
Note that the similarity distance $d_W$ depends on the graphon $W$.
The space $\overline{T}(W)$ is formed by the closure (with respect to $d_W$) of the support of $\mu$
equipped with the topology given by the metric $d_W$.
The structure of the space $\overline{T}(W)$ is related to weak regularity partitions of $W$;
in particular, if the Minkowski dimension of $\overline{T}(W)$ is $d$,
then $W$ has a weak $\varepsilon$-regular partition with $O(\varepsilon^{-d})$ parts.
We refer the reader to~\cite[Chapter 13]{bib-lovasz-book} for further details.

\subsection{Finite forcibility}

A graphon $W$ is {\em finitely forcible} if there exist graphs $H_1,\ldots,H_k$ such that
every graphon $W'$ satisfying $d(H_i,W)=d(H_i,W')$ for every $i\in[k]$ is weakly isomorphic to $W$.
For example, the result of Diaconis, Holmes, and Janson~\cite{DiHoJa09}
is equivalent to the statement that the half-graphon $W_{\triangle}(x,y)$,
which is defined as $W_{\triangle}(x,y) = 1$, if $x+y\geq 1$, and $W_{\triangle} = 0$, otherwise,
is finitely forcible.
We refer the reader to~\cite{bib-lovasz11+} for further examples of finitely forcible graphons and
to Section~\ref{sec:concl} for the discussion of some further results on finitely forcible graphons.

Following the framework from \cite{bib-rademacher}, when proving that a graphon is finitely forcible,
we give a set of {\em constraints} that uniquely determines $W$ rather than
listing the finitely many graphs and their densities that uniquely determine $W$.
A {\em constraint} is an equality between two density expressions,
where a {\em density expression} is a formal real polynomial combination of graphs,
i.e., a real number or a graph $H$ are density expressions, and
if $D_1$ and $D_2$ are two density expression,
then the sum $D_1 + D_2$ and the product $D_1\cdot D_2$ are also density expressions.
A graphon $W$ {\em satisfies} a constraint $D_1=D_2$
if both $D_1$ and $D_2$ are equal when evaluated with each $H$ substituted with $d(H,W)$.
As it was observed in~\cite{bib-rademacher},
if a graphon $W$ is a unique (up to weak isomorphism) graphon that
satisfies a finite set $\mathcal{C}$ of constraints,
then the graphon $W$ is finitely forcible.
In particular, $W$ is the unique (up to weak isomorphism) graphon with densities of graphs appearing in $\mathcal{C}$
equal to their densities in $W$. 

In~\cite{bib-rademacher},
it was also observed that a more general form of constraints, called {\em rooted constraints},
can be used to prove that a graphon is finitely forcible.
A graph is rooted
if it has $m$ distinguished vertices labeled with numbers $1,\ldots,m$;
these vertices are referred to as {\em roots} while the other vertices are {\em non-roots}.
Two rooted graphs are {\em compatible}
if the subgraphs induced by their roots are isomorphic through an isomorphism mapping the roots with the same label to each other.
Similarly, two rooted graphs are isomorphic if there exists an isomorphism mapping the $i$-th root of one of them to the $i$-th root of the other;
in particular, if two rooted graphs are isomorphic, then they are compatible.

A {\em rooted density expression} is a formal real polynomial combination of compatible rooted graphs.
We next describe how constraints formed by rooted expressions are interpreted.
Consider a graphon $W$ and a rooted graph $H$ with $m$ roots, and let $H_0$ be the graph induced by the $m$ roots of $H$.
We define the auxiliary function $c_H:[0, 1]^m \rightarrow [0,1]$;
the value of $c_H(x_1,\ldots,x_m)$ is equal to the probability that
a $W$-random graph is isomorphic to $H$ conditioned on the $m$ roots being associated with $x_1,\ldots,x_m$ (in this order), i.e.,
\begin{center}
\scalebox{0.80}{
$c_H(x_1,\ldots,x_m)=\frac{(|H|-m)!}{|\Aut(H)|}\int\limits_{(x_{m+1},\ldots,x_{|H|})\in [0,1]^{|H|-m}}\prod\limits_{(i,j)\in E(H)}W(x_i,x_j)
\prod\limits_{(i,j)\not\in E(H)}\left(1-W(x_i,x_j)\right)\dif\lambda_{|H|-m}\mbox{,}$}
\end{center}
where $\Aut(H)$ is the group of automorphisms of $H$ that preserves the roots, and
the vertices of $H$ are numbered in a way that the first $m$ vertices are the roots (in the order that they have).

Let $D=D'$ be a constraint such that $D$ and $D'$ are compatible rooted density expressions with graphs containing $m$ roots.
For every graph $H$ appearing in $D$ and $D'$, substitute the function $c_H$;
both $D$ and $D'$ can now be viewed as functions $c_D$ and $c'_D$ from $[0,1]^m$ to $[0,1]$.
We say that the graphon $W$ {\em satisfies} the constraint $D=D'$
if the functions $c_D$ and $c'_D$ are equal almost everywhere.
We comment that, on several occasions, we consider constraints containing a fraction of two rooted density expressions $D/D'$.
A constraint containing such fractions should be understood as saying that
both sides are multiplied by the denominators of all the fractions,
e.g., $D_1/D'_1=D_2/D'_2$ should be understood as $D_1\cdot D_2'= D_2\cdot D_1'$.
One of the results in~\cite{bib-rademacher} asserts that
for every two compatible rooted density expressions $D$ and $D'$,
there exist density expressions $C$ and $C'$ such that
a graphon $W$ satisfies $D=D'$ if and only if it satisfies $C=C'$.

A graphon $W$ is {\em partitioned}
if there exist $k\in\NN$, positive reals $a_1,\ldots,a_k$ summing to one and distinct reals $d_1,\ldots,d_k$ between $0$ and $1$ such that
the set of vertices of $W$ with degree $d_i$ has measure $a_i$;
we write $A_i$ for the set of vertices of degree $d_i$ for $i\in [k]$ and refer to $A_i$ as to a {\em part} of the graphon $W$.

A graph $H$ is {\em decorated} if its vertices are labeled with parts $A_1,\ldots,A_k$.
The density of a decorated graph $H$ in a graphon $W$
is the probability that the $W$-random graph is the graph $H$ 
conditioned on the event that all sampled vertices are in the parts corresponding to their labels. 
For example, if $H$ is an edge with its two vertices labeled with parts $A_1$ and $A_2$,
then the density of $H$ in $W$ is the density of edges between the parts $A_1$ and $A_2$, i.e.,
$$d(H,W)=\frac{1}{\lambda(A_1)\lambda(A_2)}\int_{A_1}\int_{A_2}W(x,y)\,\dif x\,\dif y\;\mbox{.}$$
Similarly as in the case of non-decorated graphs,
we can define rooted decorated graphs, rooted decorated density expressions and form constraints using such expressions.
A constraint that uses (rooted or non-rooted) decorated graphs is referred to as {\em decorated}.
One of the results from~\cite{bib-rademacher}, which we state as Lemma~\ref{lemma:decorated},
asserts that for every decorated constraint, there exists an equivalent ordinary constraint.

The structural properties of a graphon $W$ that satisfies a given set of constraints can be analyzed in several different ways.
The constraints of the form $D=0$ where $D$ is a single graph $G$
can be understood as forbidding $G$ as a subgraph in a $W$-random graph.
Consequently, the induced removal lemma and other combinatorial arguments
can be used to derive some structural properties of every graphon satisfying $D=0$.
However, it is also possible to derive properties of such graphons in an analytic way,
which is the way that we will generally use in our exposition.

We next introduce the convention for depicting decorated constraints used throughout the paper;
an example of the use of this convention can be found in Figure~\ref{fig:XxE}.
The roots of decorated graphs will be depicted by squares and non-root vertices by circles;
all vertices will be labeled by the names of the corresponding parts of a graphon.
The full lines connecting vertices correspond to edges and dashed lines to non-edges.
No connection between a pair of vertices represents that both edge or non-edge are allowed between the vertices,
i.e.,
the corresponding density expression should be understood as the sum of the expressions containing
the graph with and without such the edge (unless the edge is missing between two roots).
For example, if three pairs of vertices are missing a connection,
the density expression is the sum of all eight graphs that can be obtained by including or not including the edge between the three pairs.
If the edge is missing between two roots, then the density constraint is required to hold
both when the edge is included between the pair of root vertices in all graphs and when it is included in no graph.
To avoid any possible ambiguity with interpretations of the drawings of rooted constraints,
the positions of the roots of all graphs appearing in a rooted decorated density constraint will always be identical (see Figure~\ref{fig:distr} for an example).

We conclude this section by explicitly stating three lemmas that were proven in~\cite{bib-rademacher} and that we use further.
The first lemma guarantees the existence of a set of constraints that
force a graphon satisfying these constraints to be a partitioned graphon with a given partition and given degrees.
\begin{lemma}\label{lemma:partitioned}
\label{lemma:partition}
Let $k\in\NN$, $a_1,\ldots,a_k$ be positive real numbers summing to one and
let $d_1,\ldots$, $d_k$ be distinct reals between $0$ and $1$.
There exists a finite set of constraints $\mathcal{C}$ such that
a graphon $W$ satisfies $\mathcal{C}$ if and only if
$W$ is a partitioned graphon with $k$ parts such that the $i$-th part has measure $a_i$ and its vertices have degree $d_i$.
\end{lemma}

The following lemma says that decorated constraints have the same expressing power as non-decorated constraints.
\begin{lemma}\label{lemma:decorated}
Let $k\in\NN$, let $a_1,\ldots,a_k$ be positive real numbers summing to one, and
let $d_1,\ldots,d_k$ be distinct reals between zero and one.
Further, let $D_1$ and $D_2$ be two compatible rooted decorated density expressions with decorations $A_1,\ldots,A_k$.
There exist an ordinary density expression $D$, i.e., $D$ has no roots and no decorations, such that
every partitioned graphon $W$ with $k$ parts formed by vertices of degree $d_i$ and measure $a_i$ each
satisfies $D_1=D_2$ if and only if it satisfies $D=0$.
\end{lemma}
We remark that our definition of interpreting decorated density expressions differ from the definition given  in~\cite{bib-rademacher}.
However, the difference results only in a constant multiplicative factor depending on the measures of the parts of a graphon;
in particular, Lemma~\ref{lemma:decorated} also holds with the definition of decorated constraints that we use.

The last lemma states that there exists a finite set of constraints guaranteeing that
a partitioned graphon is constant between a specific pair of its parts.

\begin{lemma}\label{lemma:pseudorandom} 
For all $k\in \mathbb{N}$, positive reals $a_1, \ldots, a_k$ summing to one, distinct reals $d_1, \ldots, d_k$ between zero and one,
$\ell,\ell'\leq k$, $l\not=l'$, and $p\in[0,1]$,
there exists a finite set of constraints $\mathcal{C}$ such that
every partitioned graphon $W$ with $k$ parts $A_1,\ldots,A_k$ such that the measure of $A_i$ is $a_i$ and
all vertices of $A_i$ have degrees $d_i$
satisfies $\mathcal{C}$ if and only if $W(x,y)=p$ for almost every $x \in A_\ell$ and $y\in A_{\ell'}$.
\end{lemma}

\section{The \name graphon}

In this section, we define the graphon from Theorem~\ref{thm:main};
the graphon is denoted by $\graphon$ and referred to as the {\em \name graphon}.
For convenience, we provide a sketch of the structure of the graphon $\graphon$ in Figure~\ref{fig:graphon}.

\begin{figure}
\centering
\includegraphics[scale=1]{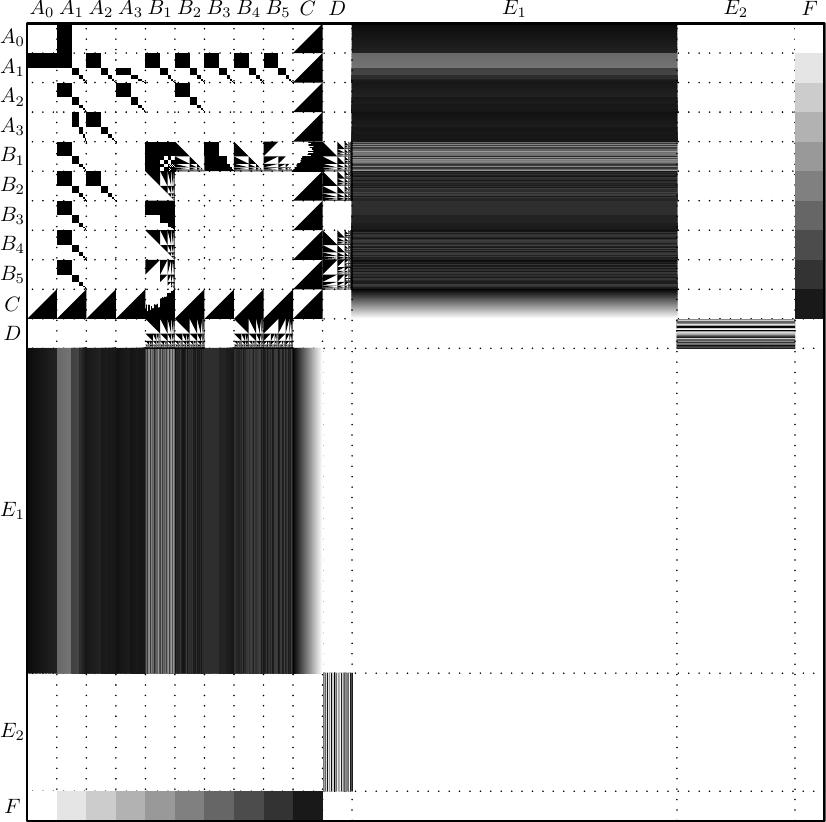}
\caption{The \name graphon. The origin of the coordinate system is in the top left corner; the values of the graphon are visualized using different shades of gray (with white being zero and black being one). The graphon between the parts $X$ and $Y$, $X\in\{\B_1,\D\}$ and $Y\in\{\B_1, \B_2, \B_4, \B_5\}$, is drawn in an imprecise simplified way because of the complex structure. To simplify the picture, the parts are labeled by their names without the superscripts.}
\label{fig:graphon}
\end{figure}
 
The {\em \name graphon} $\graphon$ is a partitioned graphon with $14$ parts,
which are denoted by $\A_0, \ldots, \A_3, \B_1,\ldots, \B_5,\C,\D,\E_1,\E_2,\F$.
Each part has measure $1/27$ except for the parts $\E_1$ and $\E_2$ that have measure $11/27$ and $4/27$, respectively.
The degrees of the vertices in the parts are listed in Table~\ref{tab:degrees}.
We will not compute the exact values $e_1$ and $e_2$ of the degrees of vertices in $\E_1$ and $\E_2$, respectively;
however, the definition of the graphon will imply that $e_1\in (4.5/27,10/27)$ and $e_2<1/27$.
In particular, vertices in different parts have different degrees.
The high level overview of the roles of individual parts of the \name graphon
can be found at the end of Section~\ref{sec:intro}.

\begin{table}
\begin{center}
\scalebox{0.90}{
\begin{tabular}{| l |>{$}c<{$}|>{$}c<{$}|>{$}c<{$}|>{$}c<{$}|>{$}c<{$}|>{$}c<{$}|>{$}c<{$}|>{$}c<{$}|>{$}c<{$}|>{$}c<{$}|>{$}c<{$}|>{$}c<{$}|>{$}c<{$}|>{$}c<{$}|}
    \hline
    part & \A_0 & \A_1 & \A_2 & \A_3 & \B_1 & \B_2 & \B_3 & \B_4 & \B_5 & \C & \D & \E_1 & \E_2 & \F \\ \hline 
    degree &  \frac{110}{270} &\frac{111}{270} & \frac{112}{270} & \frac{113}{270} & \frac{114}{270} & \frac{115}{270} & \frac{116}{270} & \frac{117}{270} & \frac{118}{270} & \frac{119}{270} & \frac{40}{270} & e_1 & e_2 & \frac{45}{270} \\
    \hline
\end{tabular}
}
\end{center}
\caption{The degrees of the vertices in the parts of the graphon $\graphon$.}
\label{tab:degrees}
\end{table}

We describe the graphon $\graphon$ as a collection of functions $\graphon^{X\times Y}$ on products of the parts $\X$ and $\Y$.
To simplify our exposition,
we define these as functions from $[0,1]^2$ to $[0,1]$,
assuming that we have a fixed measurable bijection $\eta_X$ from each part $\X$ to $[0,1]$ such that
$\lambda\left(\eta^{-1}_X(S)\right)=\lambda(S)\lambda(\X)$ for every measurable set $S\subseteq [0,1]$.
So, it holds $\graphon(x,y)=\graphon^{X\times Y}(\eta_X(x),\eta_Y(y))$ for $x\in\X$ and $y\in\Y$,
i.e., the graphon $\graphon$ consists of appropriately scaled functions $\graphon^{X\times Y}$.
Note that, unlike graphons,
the functions $\graphon^{X\times Y}$ need not to be symmetric if $X\not=Y$;
however, these functions satisfy $\graphon^{X\times Y}(x,y)=\graphon^{Y\times X}(y,x)$.

We now introduce additional notation used in the definition of the graphon $\graphon$ and in the proof.
For $x\in [0,1)$, let $\twoblock{x}$ be such $k\in\NN$ that $x\in [1-2^{-k+1},1-2^{-k})$ and
let $\re{x}=(x-(1-2^{-k+1}))\cdot 2^{k}$.
Informally speaking, we imagine $[0,1]$ as partitioned into consecutive intervals of measures $1/2$, $1/4$, etc., and
$\twoblock{x}$ indicates the index of the interval that $x$ belongs to and
$\re{x}$ is the relative position of $x$ within this interval.
Observe that $x=1-2^{1-\twoblock{x}}+\re{x}/2^{\twoblock{x}}$ for every $x\in [0,1)$.
Using this notation,
we define the {\em \checkered} function $\varkappa:[0,1]^2\rightarrow [0,1]$ as follows (see Figure~\ref{fig:ch}):
$$\varkappa(x,y)=
\left\{ \begin{array}{rl}
1 & \mbox{if }  \twoblock{x}=\twoblock{y}\\ 
0 & \mbox{otherwise.} \\
\end{array}\right.$$

\begin{figure}[b]
\centering
\includegraphics{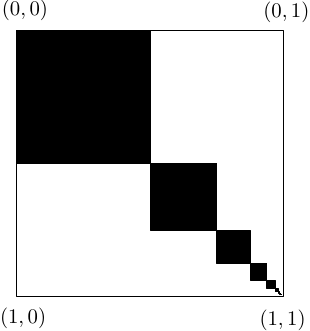}
\caption{The \checkered\ function $\varkappa$.}
\label{fig:ch}
\end{figure}

We are now ready to start with defining the structure between different parts of the graphon $\graphon$.

\begin{description}

\item[$\graphon^{A_0\times A_1}(x,y)=$]
$\left\{ \begin{array}{rl}
1 & \mbox{for } (x,y)\in [0,1]\times[0,1/2]\mbox{, and}\\ 
0 & \mbox{otherwise.} \\
\end{array}\right.$

\item[$\graphon^{A_1\times A_1}=$]
$\graphon^{A_1\times A_2}=\graphon^{A_1\times B_1}=\graphon^{A_1\times B_2}=\graphon^{A_1\times B_3}=\graphon^{A_1\times B_4}=\graphon^{A_1\times B_5}$
$=\graphon^{A_2\times A_3}=\graphon^{A_2\times B_2}=\varkappa.$

\end{description}

For $X\in\{ A_0, \ldots, A_3, B_2,\ldots, B_5, C\}$, let:

\begin{description}
\item[$\graphon^{C\times X}(x,y)=$]
$\left\{ \begin{array}{rl}
1 & \mbox{for } x+y\geq 1\mbox{, and}\\ 
0 & \mbox{otherwise.} \\
\end{array}\right.$
\end{description}

The rest of the definition of the graphon $\graphon$ depends on a collection of measure preserving functions,
which we call a \recipe.
A {\em \recipe} $\ccube$ is a set of measure preserving maps $\cube_n$ for $n\in\nati$ such that
$\cube_n:[0,1]\rightarrow [0,1]^{n}$.
Recall that $\nati=\NN\cup\{\infty\}$ and so
we understand $[0,1]^\infty$ to be $[0,1]^\NN$.
An example of a recipe is a collection of maps that ``zip'' the standard binary representations of $x_i$,
i.e., the digits of $\cube_n(x)$ on the positions congruent to $i$ modulo $n$ are determined by the digits of $x_i$, $i\in [n]$, and
the digits of $\cube_{\infty}(x)$ on the positions congruent to $2^{i-1}$ modulo $2^i$ are determined by the digits of $x_i$, $i\in\NN$.
Observe that $\ccube=\{\cube_{n}|n\in \nati\}$ is a recipe if and only if
\begin{equation}
  \lambda\left(\{x|\forall\ i\in[n]\ (\cube_{n}(x))_i\leq z_i\}\right)=\prod_{i=1}^n z_i\mbox{ for every }(z_1,\dots z_n)\in[0,1]^n
\label{eq2a}
\end{equation}
for every $n\in\NN$ and 
\begin{equation}
  \lambda\left(\{x|\forall\ i\in[k]\ (\cube_{\infty}(x))_i\leq z_i\}\right)=\prod_{i=1}^k z_i\mbox{ for every }(z_1,\dots z_k)\in[0,1]^k
\label{eq2b}
\end{equation}
for every $k\in\NN$, where $(x)_i$ is the $i$-th coordinate of $x\in [0,1]^n$, $n\in\nati$.
A \recipe\ is {\em bijective} if all the maps $\cube_n$, $n\in\nati$, are bijective.

For the rest of the definition of the graphon $\graphon$, we fix a bijective \recipe\ $\ccube$.
It can be shown that the definition of $\graphon$ does not depend on this choice in the sense that 
the graphons defined for different choices of $\ccube$ are weakly isomorphic (this statement stays
true even if $\ccube$ is a \recipe\ that is not bijective).

\begin{description}

\item[$\graphon^{A_1\times A_3}(x,y)=$]
$\left\{ \begin{array}{rl}
1 & \mbox{if } \twoblock{x}=\twoblock{y}+1\mbox{, and}\\
0 & \mbox{otherwise.} \\
\end{array}\right.$

\item[$\graphon^{C\times B_1}(x,y)=$]
$\left\{ \begin{array}{rl}
1 & \mbox{for } (1-2^{1-\twoblock{y}})+(\cube_{\twoblock{y}}(\re{y}))_{1}\cdot 2^{-\twoblock{y}}+x\geq 1\mbox{, and}\\ 
0 & \mbox{otherwise.} \\
\end{array}\right.$

\item[$\graphon^{B_1\times B_1}(x,y)=$]
$\left\{ \begin{array}{rl}
1 & \mbox{if }(\cube_{\twoblock{x}}(\re{x}))_{k} \leq (\cube_{\twoblock{y}}(\re{y}))_{k} \mbox{ for every } k\leq\min (\twoblock{x},\twoblock{y} ),\\
1 & \mbox{if }(\cube_{\twoblock{x}}(\re{x}))_{k} \geq (\cube_{\twoblock{y}}(\re{y}))_{k} \mbox{ for every } k\leq\min (\twoblock{x},\twoblock{y} )\mbox{, and}\\
0 & \mbox{otherwise.} \\
\end{array}\right.$

\item[$\graphon^{B_1\times B_2}(x,y)=$]
$\left\{ \begin{array}{rl}
1 & \mbox{if } \twoblock{x}\geq\twoblock{y} \mbox{ and } \re{y} \leq (\cube_{\twoblock{x}}(\re{x}))_{\twoblock{y}}\mbox{, and}\\ 
0 & \mbox{otherwise.} \\
\end{array}\right.$

\item[$\graphon^{B_1\times B_3}(x,y)=$]
$\left\{ \begin{array}{rl}
1 & \mbox{if } \twoblock{x}\geq\twoblock{y}\mbox{, and}\\ 
0 & \mbox{otherwise.} \\
\end{array}\right.$

\item[$\graphon^{B_1\times B_4}(x,y)=$]
$\left\{ \begin{array}{rl}
1 & \mbox{if } \twoblock{x}\geq\twoblock{y} \mbox{ and } \re{y} \leq \prod\limits_{i=1}^{\twoblock{y}}(\cube_{\twoblock{x}}(\re{x}))_{i}\mbox{, and}\\ 
0 & \mbox{otherwise.} \\
\end{array}\right.$

\item[$\graphon^{B_1\times B_5}(x,y)=$]
$\left\{ \begin{array}{rl}
1 & \mbox{if } \twoblock{x}\geq\twoblock{y} \mbox{ and } \re{y} \leq \prod\limits_{i=1}^{\twoblock{y}}(1-(\cube_{\twoblock{x}}(\re{x}))_{i})\mbox{, and}\\ 
0 & \mbox{otherwise.} \\
\end{array}\right.$

\item[$\graphon^{D\times B_1}(x,y)=$]
$\left\{ \begin{array}{rl}
1 & \mbox{if }\cube_{\infty}(\re{y})_k\leq(\cube_{\infty}(x))_{k}\mbox{ for every }\ k\leq \twoblock{y}\mbox{, and}\\ 
0 & \mbox{otherwise.} \\
\end{array}\right.$

\item[$\graphon^{D\times B_2}(x,y)=$]
$\left\{ \begin{array}{rl}
1 & \mbox{if } \re{y} \leq (\cube_{\infty}(x))_{\twoblock{y}}\mbox{, and}\\ 
0 & \mbox{otherwise.} \\
\end{array}\right.$

\item[$\graphon^{D\times B_4}(x,y)=$]
$\left\{ \begin{array}{rl}
1 & \mbox{if } \re{y} \leq \prod\limits_{i=1}^{\twoblock{y}}(\cube_{\infty}(x))_{i}\mbox{, and}\\ 
0 & \mbox{otherwise.} \\
\end{array}\right.$

\item[$\graphon^{D\times B_5}(x,y)=$]
$\left\{ \begin{array}{rl}
1 & \mbox{if } \re{y} \leq \prod\limits_{i=1}^{\twoblock{y}}(1-(\cube_{\infty}(x))_{i})\mbox{, and}\\ 
0 & \mbox{otherwise.} \\
\end{array}\right.$
\end{description}

\noindent For every $X\in\{ A_0, \ldots, A_3, B_1, \ldots, B_5, C\}$, we set:

\begin{description}
\item[$\graphon^{E_1\times X}(x,y)=$]
$1-1/11\sum\limits_{Y\in \A_0, \ldots, \A_3, \B_1,\ldots, \B_5, \C, \D} \deg_{Y} y.$
\end{description}

\noindent We further define
\begin{description}
\item[$\graphon^{E_2\times D}(x,y)=$]
$1-1/4\sum\limits_{Y\in \B_1,\B_2,\B_4, \B_5} \deg_{Y} y .$

\item[$\graphon^{F\times A_1}(x,y)=1/10$] for all $(x,y)\in [0,1]^2$,
\item[$\graphon^{F\times A_2}(x,y)=2/10$] for all $(x,y)\in [0,1]^2$,
\item[$\graphon^{F\times A_3}(x,y)=3/10$] for all $(x,y)\in [0,1]^2$,
\item[$\graphon^{F\times B_1}(x,y)=4/10$] for all $(x,y)\in [0,1]^2$,
\item[$\graphon^{F\times B_2}(x,y)=5/10$] for all $(x,y)\in [0,1]^2$,
\item[$\graphon^{F\times B_3}(x,y)=6/10$] for all $(x,y)\in [0,1]^2$,
\item[$\graphon^{F\times B_4}(x,y)=7/10$] for all $(x,y)\in [0,1]^2$,
\item[$\graphon^{F\times B_5}(x,y)=8/10$] for all $(x,y)\in [0,1]^2$, and
\item[$\graphon^{F\times C}(x,y)=9/10$] for all $(x,y)\in [0,1]^2$.
\end{description}

If we have defined a function $\graphon^{X\times Y}$,
we set $\graphon^{Y\times X}(x,y)=\graphon^{X\times Y}(y,x)$.
Finally,
the graphon $\graphon$ is equal to 0 between parts $\X$ and $\Y$ such that
we have not defined a function $\graphon^{X\times Y}$ or $\graphon^{Y\times X}$.
This completes the definition of the graphon $\graphon$.

We now argue that $e_1\in (4.5/27,10/27)$ and $e_2<1/27$.
Let $x\in\E_1$.
Since $N(x)$ is a subset of $\A_0\cup\dots\cup\A_3\cup\B_1\cup\cdots\cup\B_5\cup\C$,
the measure of $N(x)$ is at most $10/27$.
Since it does not hold that $\graphon(x,y)=1$ for almost all $y\in N(x)$,
we get that $e_1<10/27$.
Observe that it holds for every $X\in\{\A_0,\ldots,\A_3,\B_2,\ldots,\B_5,\C\}$ that
$\deg_{\E_1}(x)>1/2$ for every $x\in X$.
It follows that $e_1>4.5/27$.
Similarly, $N(x)$ is a subset of $\D$ for every $x\in\E_2$ and
it does not hold that $\graphon(x,y)=1$ for almost all $y\in N(x)$;
this implies that $e_2<1/27$.

Before proceeding further,
we introduce additional notation related to splitting parts $A_i$, $i\in\{1,2,3\}$, and
$B_j$, $j\in\{1,\ldots,5\}$, into smaller pieces.
For $i\in\{1,2,3\}$, the set of vertices $x\in\A_i$ with $\deg_{\A_1} x=2^{-k}$ is denoted by $\A_{i,k}$ and
$\A_{i,k}$ is called the {\em $k$-th level} of $\A_i$.
Similarly, $\B_{j,k}$, $j\in\{1,\ldots,5\}$, is the set of vertices $x\in\B_j$ such that $\deg_{\A_1} x=2^{-k}$.
Note that measure of the $k$-th level $\A_{i,k}$ is $2^{-k}/27$; the same holds for $\B_{j,k}$.

\subsection{Dimension of the space of typical vertices}

We finish this section with showing that both $T(\graphon)$ and $\overline{T}(\graphon)$ have infinite dimension.

\begin{proposition}
\label{prop:infinite}
Both $T(\graphon)$ and $\overline{T}(\graphon)$ contain a subspace homeomorphic to $[0,1]^{\infty}$.
\end{proposition}

\begin{proof}
Observe that every vertex contained in $\D$ is typical (both with respect to $T(\graphon)$ and
with respect to $\overline{T}(\graphon)$) and
define a map $h:\D\to [0,1]^\infty$ as
$$h(x)=\left(\deg^{\graphon}_{\B_{2,i}}x\right)_{i\in\NN}\;\mbox{.}$$
Because $\cube_{\infty}$ is a bijection,
$h$ is a bijection between $\D$ and $[0,1]^\infty$.
We next show that $h^{-1}$ is continuous when $\D$ is equipped with the topology of the space $T(\graphon)$.
To do so, we need to bound the $L^1$-distance of the functions $f^{\graphon}_x$ and $f^{\graphon}_{x'}$
in terms of $h(x)$ and $h(x')$ for all $x,x'\in\D$, where $f^{\graphon}_x(y):=\graphon(x,y)$.

First note that 
\[\deg^{\graphon}_{\B_{1,i}}x=\deg^{\graphon}_{\B_{4,i}}x=\prod_{k\in[i]}\deg^{\graphon}_{\B_{2,k}}x\quad \mbox{ and}\quad
  \deg^{\graphon}_{\B_{5,i}}x=\prod_{k\in[i]}(1-\deg^{\graphon}_{\B_{2,k}}x)\]
for every $x\in\D$.
The value of $||f^{\graphon}_x-f^{\graphon}_{x'}||_1$ is the sum of
the corresponding integrals over $y$ from $\B_1$, $\B_2$, $\B_4$, $\B_5$ and $\E_2$.
The term corresponding to the integral over $y$ from $\B_2$ is equal to
\[\sum_{i=1}^\infty\lambda(\B_{2,i})\left|\deg^{\graphon}_{\B_{2,i}}x-\deg^{\graphon}_{\B_{2,i}}x'\right|\;\mbox{,}\]
the term corresponding to the integral over $y$ from $\B_4$ is equal to
\[\sum_{i=1}^\infty\lambda(\B_{4,i})\left|\prod_{k=1}^{i}\deg^{\graphon}_{\B_{2,k}}x-\prod_{k=1}^{i}\deg^{\graphon}_{\B_{2,k}}x'\right|\;\mbox{,}\]
and the term corresponding to the integral over $y$ from $\B_5$ is equal to
\[\sum_{i=1}^\infty\lambda(\B_{5,i})\left|\prod_{k=1}^{i}(1-\deg^{\graphon}_{\B_{2,k}}x)-\prod_{k=1}^{i}(1-\deg^{\graphon}_{\B_{2,k}}x')\right|\;\mbox{.}\]
The term corresponding to the integral over $y$ from $\B_1$ is at most
\[\sum_{i=1}^\infty\lambda(\B_{1,i})\sum_{k=1}^i\left|\deg^{\graphon}_{\B_{2,k}}x-\deg^{\graphon}_{\B_{2,k}}x'\right|\;\mbox{.}\]
We next observe that
\[\left|\prod_{k=1}^{i}\deg^{\graphon}_{\B_{2,k}}x-\prod_{k=1}^{i}\deg^{\graphon}_{\B_{2,k}}x'\right|\le
  \sum_{k=1}^i\left|\deg^{\graphon}_{\B_{2,k}}x-\deg^{\graphon}_{\B_{2,k}}x'\right|\mbox{ and}\]
\[\left|\prod_{k=1}^{i}(1-\deg^{\graphon}_{\B_{2,k}}x)-\prod_{k=1}^{i}(1-\deg^{\graphon}_{\B_{2,k}}x')\right|\le
  \sum_{k=1}^i\left|\deg^{\graphon}_{\B_{2,k}}x-\deg^{\graphon}_{\B_{2,k}}x'\right|\]
for every $i\in\NN$.  
Since it holds that $\lambda(\B_{j,i})=2^{-i}/27$ for $j\in\{1,2,4,5\}$, 
we obtain that the sum of the terms corresponding to the integrals over $y$ from $\B_1$, $\B_2$, $\B_4$ and $\B_5$
is at most
\[\frac{1}{27}\left(\sum_{i=1}^\infty 2^{-i}\left|\deg^{\graphon}_{\B_{2,i}}x-\deg^{\graphon}_{\B_{2,i}}x'\right|
   + 3\sum_{i=1}^\infty 2^{-i}\sum_{k=1}^{i}\left|\deg^{\graphon}_{\B_{2,k}}x-\deg^{\graphon}_{\B_{2,k}}x'\right|\right)\;\mbox{,}\]
which is equal to
\[\frac{7}{27}\left(\sum_{i=1}^\infty 2^{-i}\left|\deg^{\graphon}_{\B_{2,i}}x-\deg^{\graphon}_{\B_{2,i}}x'\right|\right)\;\mbox{.}\]
Since the term corresponding to the integral over $y$ from $\E_2$
is at most the sum of the terms to the integrals over $y$ from $\B_1$, $\B_2$, $\B_4$ and $\B_5$,
we conclude that
\[||f^{\graphon}_x-f^{\graphon}_{x'}||_1\leq\frac{14}{27}\left(\sum_{i=1}^\infty 2^{-i}\left|\deg^{\graphon}_{\B_{2,i}}x-\deg^{\graphon}_{\B_{2,i}}x'\right|\right)\;\mbox{.}\]
It follows that $h^{-1}$ is a continuous map from $[0,1]^\infty$ to $\D$.
Since $h^{-1}$ is a continuous injective map from a compact space to a Haussdorf space,
it follows that $h$ is a homeomorphism between $\D$ with the topology given by $T(\graphon)$ and $[0,1]^\infty$.
Since the identity map from $T(\graphon)$ to $\overline{T}(\graphon)$ is injective and continuous~\cite{bib-lovasz10++},
it also follows that $h$ is a homeomorphism between $\D$ with the topology given by $\overline{T}(\graphon)$ and $[0,1]^\infty$.
\end{proof}

\section{Constraints}
\label{sec:list}

This section and the next section are devoted to the proof of the following theorem,
which together with Proposition~\ref{prop:infinite} implies Theorem~\ref{thm:main}.

\begin{theorem}\label{thm:forcible}
The \name graphon $\graphon$ is finitely forcible.
\end{theorem}

In this section,
we present the set $\constr$ of the constraints that such that
the graphon $\graphon$ is the unique graphon satisfying $\constr$.
We only list the constraints contained in $\constr$ and
their analysis is postponed to the next section.

We present the constraints contained in the set $\constr$ split into groups
depending on the properties of a graphon that they force, and
we informally describe these properties.

\begin{description}
\item{\em Partition constraints} are the constraints given in the Lemma~\ref{lemma:partitioned},
     which are satisfied by partitioned graphons with the same number of parts as $\graphon$ and
     with the measures and the degrees of vertices of the parts as in $\graphon$.
\end{description}

All the constraints that are presented in the rest are decorated constraints
with vertices labelled by the parts $A_0,\ldots, A_3, B_1,\ldots, B_5, C,D,E_1, E_2,F$.
 
\begin{description}
\item{{\em The zero constraints}} force that $W$ equals $0$ almost everywhere on
\begin{itemize}
\item $A_0\times (A_0\cup A_2\cup A_3\cup B_1\cup B_2\cup B_3\cup B_4\cup B_5\cup D\cup E_2\cup F)$,
\item $A_1\times (D\cup E_2)$,
\item $A_2 \times (A_2\cup B_1\cup B_3\cup B_4\cup B_5\cup D\cup E_2)$,
\item $A_3 \times (A_3\cup B_1\cup\cdots\cup B_5\cup D\cup E_2)$,
\item $B_2 \times (B_2\cup \cdots\cup B_5\cup E_2)$,,
\item $B_3 \times (B_3\cup B_4\cup B_5\cup D\cup E_2)$,
\item $B_4 \times (B_4\cup B_5\cup E_2)$,
\item $B_5 \times (B_5\cup E_2)$,
\item $C\times (D\cup E_2)$,
\item $D \times (D\cup E_1)$,
\item $E_1\times (E_1\cup E_2\cup F)$,
\item $E_2 \times (E_2\cup F)$, and
\item $F \times F$.
\end{itemize}
The constraint forcing the zero edge density between parts $X$ and $Y$ is depicted in Figure~\ref{fig:zero}.

\begin{figure}
\centering
\includegraphics{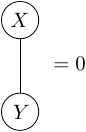}
\caption{Constraint forcing zero edge density.}
\label{fig:zero}
\end{figure}

\item{{\em The degree unifying constraints}}
force that the relative degree
of almost every vertex $x$ from a part $A_i$, $i=0,\ldots,3$, a part $B_j$, $j=1,\ldots,5$ and the part $C$
with respect to the complement of $E_2\cup F$, i.e. $A_0\cup\cdots\cup A_3\cup B_1\cup\cdots\cup B_5\cup C\cup D\cup E_1$,
is equal to $1/2$, and
that $W(x,z)$ is constant for almost every such $x$ when $z$ ranges through the part $E_1$.
These constraints also force that the degree of almost every vertex $y$ from the part $D$ is $4/27$ and
$W(y,z)$ is constant for almost every such $y$ when $z$ ranges through the part $E_2$.
The constraints are depicted in Figures~\ref{fig:XxE} and~\ref{fig:XxE2}.

\begin{figure}
\centering
\includegraphics{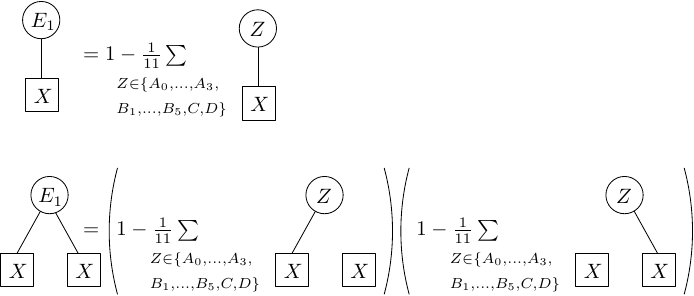}
\caption{The degree unifying constraints contain the depicted constraints for all the choices of $X$ in $\{ A_0, \ldots, A_3, B_1,\ldots, B_5, C\}$.}
\label{fig:XxE}
\end{figure}

\begin{figure}
\centering
\includegraphics{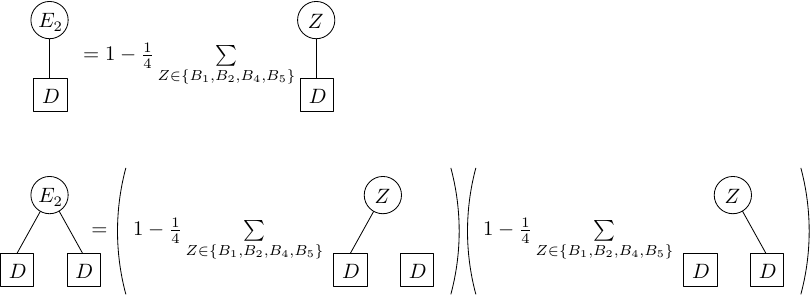}
\caption{The degree unifying constraints for $D$.}
\label{fig:XxE2}
\end{figure}

\item{{\em The degree distinguishing constraints}} force that
the graphon is constant between the part $F$ and each of the parts $A_0,\ldots,A_3$, $B_1,\ldots,B_5$, $C$ and $D$, and
that this constant is equal to the value given in Table~\ref{tab:FxX}.
The existence of finitely many such constraints follows from Lemma~\ref{lemma:pseudorandom};
Figure~\ref{fig:distinguish} contains an example of two constraints that
can be used to force the graphon to be equal to $9/10$ between the parts $C$ and $F$.
 
\begin{table}
\begin{center}
	\begin{tabular}{| l |>{$}c<{$}|>{$}c<{$}|>{$}c<{$}|>{$}c<{$}|>{$}c<{$}|>{$}c<{$}|>{$}c<{$}|>{$}c<{$}|>{$}c<{$}|>{$}c<{$}|>{$}c<{$}|}
    \hline
    Part & A_0 & A_1 & A_2 & A_3 & B_1 & B_2 & B_3 & B_4 & B_5 & C \\ \hline 
    Density & 0 & \frac{1}{10} & \frac{2}{10} & \frac{3}{10} & \frac{4}{10} & \frac{5}{10} & \frac{6}{10} & \frac{7}{10} & \frac{8}{10} & \frac{9}{10} \\ 
    \hline
    \end{tabular}
    \caption{Densities between the part $F$ and the other parts.}
    \label{tab:FxX}
\end{center}
\end{table}

\begin{figure}
\centering
\includegraphics{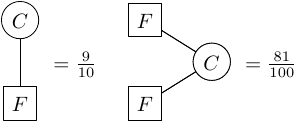}
\caption{The degree distinguishing constraints for $F\times C$.}
\label{fig:distinguish}
\end{figure}

\item{{\em The triangular constraints}} 
force that the structure of the subgraphon induced by $C\times X$ is
the same in $\graphon$ for every $X\in A_0, \ldots, A_3, B_1, \ldots, B_5, C$,
i.e., that the subgraphon induced by $C\times X$ is the half-graphon.
Let $H_i$ and $d_i$ be the finitely many graphs and their densities that
are satisfied by the half-graphon only; such a finite set of graphs exists
since the half-graphon is finitely forcible~\cite{DiHoJa09,bib-lovasz11+}.
The structure of the subgraphon induced by $C\times C$ is forced by the constraints $H'_i=d_i$
where $H'_i$ is the decorated graph obtained from $H_i$ by labeling each vertex with $C$, and
the structure of the subgraphon induced by $C\times X$ for $X\not=C$
is forced by the constraints depicted in Figure~\ref{fig:triangular}.

\begin{figure}
\centering
\includegraphics{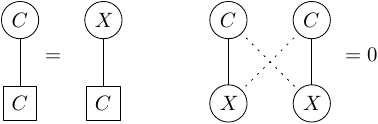}
\caption{The triangular constraints include the depicted constraints for all the choices of $X$ in $\{A_0, \ldots, A_3, B_1,\ldots, B_5\}$.}
\label{fig:triangular}
\end{figure} 

\item{{\em The main \checkered\ constraints}} 
force the \checkered\ structure of the subgraphon induced by $A_1\times A_1$. They are depicted in Figure~\ref{fig:checkered}.

\begin{figure}
\centering
\includegraphics{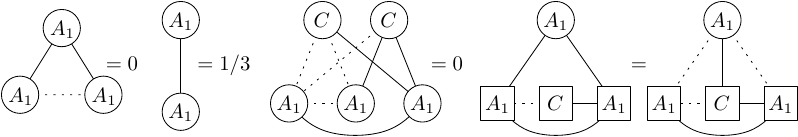}
\caption{The main \checkered\ constraints.}
\label{fig:checkered}
\end{figure}

\item{\em The complete bipartition constraints}
 force, in particular, that the subgraphons induced by $A_1\times A_2$, $A_1\times A_3$, $A_1\times B_1$, \ldots, $A_1\times B_5$, $A_2\times A_3$ and $A_2\times B_2$
 are unions of complete bipartite subgraphons. The constraints are given in Figure~\ref{fig:bipartite}.
 
\begin{figure}

\centering
\includegraphics{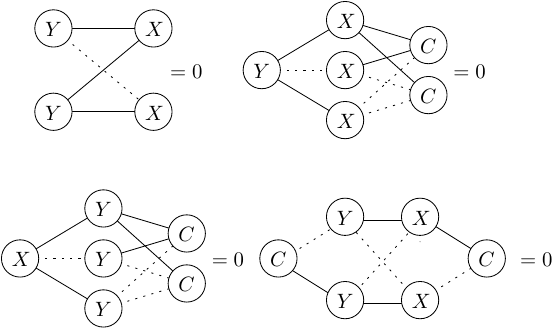}
\caption{The complete bipartition constraints consist of the top two constraint for $(X,Y)\in \{(A_1, A_2), (A_1, A_3), (A_1, B_1), \ldots, (A_1,B_5), (A_2, A_3), (A_2,B_2)\}$ and the bottom two constraints for $(X,Y)\in \{(A_1, A_2)$, $(A_1, A_3),$ $(A_1, B_2), \ldots,$ $(A_1,B_5),$ $(A_2, A_3), (A_2,B_2)\}$.}
\label{fig:bipartite}
\end{figure}

\item{{\em The auxiliary \checkered\ constraints}} 
determine the sizes of the sides of complete bipartite subgraphons in $A_1\times A_2$, $A_1\times A_3$, $A_1\times B_1, \ldots, A_1\times B_5$, $A_2\times A_3$ and $A_2\times B_2$. They are depicted in Figure~\ref{fig:aux-check}.

\begin{figure}

\centering
\includegraphics{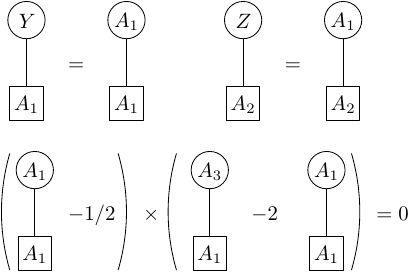}
\caption{The auxiliary \checkered\ constraints consist of the depicted constraints, where $Y$ in the first constrains attains all values in $\{A_2, B_1, \ldots, B_5\}$ and $Z$ in the second constraint attains all values in $\{A_3, B_2\}$.}
\label{fig:aux-check}
\end{figure}

\item{{\em The first level constraints}}
 force the structure of subgraphon induced by $A_0\times A_1$ and they are depicted in Figure~\ref{fig:forceA0}.
 
\begin{figure}
\centering
\includegraphics{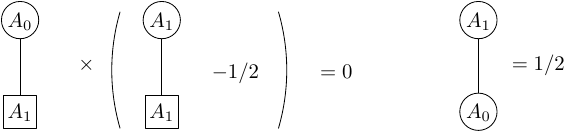}
\caption{The first level constraints.}
\label{fig:forceA0}
\end{figure}

\item{{\em The stair constraints}}
 force the structure of subgraphon induced by $B_1\times B_3$. They are depicted in Figure~\ref{fig:stairs}.
 
\begin{figure}
\centering
\includegraphics{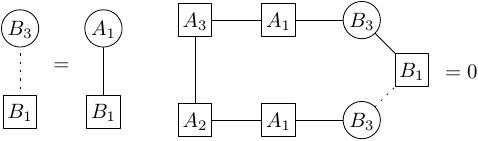}
\caption{The stair constraints.}
\label{fig:stairs}
\end{figure}

\item{{\em The coordinate constraints}}
 force some properties of the structure of the subgraphons induced by $B_1\times (B_2\cup B_4\cup B_5)$ and $D\times (B_2\cup B_4\cup B_5)$. They can be found in Figure~\ref{fig:coord}.
 
\begin{figure}

\centering
\includegraphics{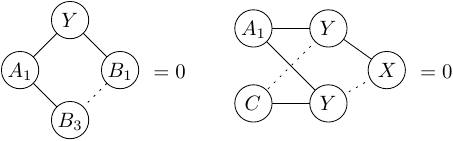}
\caption{The coordinate constraints consist of the depicted constraints, where $X$ and $Y$ attain all values in $\{B_1, D\}$ and $\{B_2, B_4, B_5\}$, respectively.}
\label{fig:coord}
\end{figure}

\item{{\em The initial coordinate constraint}}
 determines the relative degrees of vertices of $B_1$ in a subset of $B_2$. It is depicted in Figure~\ref{fig:first}.

\begin{figure}
\centering
\includegraphics{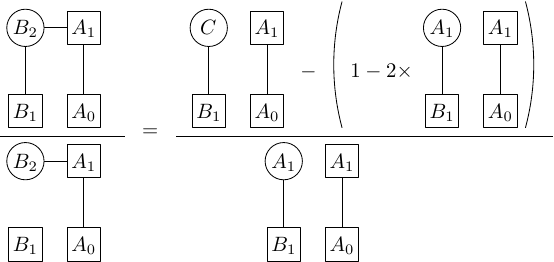}
\caption{The initial coordinate constraint.}
\label{fig:first}
\end{figure}

\item{{\em The distribution constraints}}
 determine the relative degrees of vertices of $B_2$ in $B_1$ and $D$ are depicted in Figure~\ref{fig:distr}.
 
\begin{figure}
\centering
\includegraphics{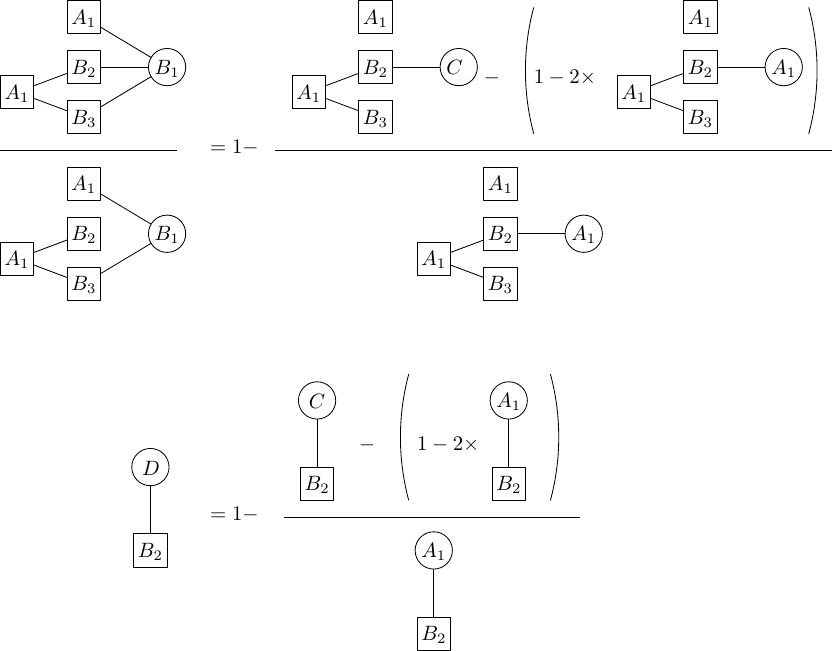}
\caption{The distribution constraints.}
\label{fig:distr}
\end{figure}

\item{\em The product constraints}
 force the structure of the subgraphons induced by $B_1\times B_4$, $D\times B_4$, $B_1\times B_5$ and $D\times B_5$. They are depicted in Figures~\ref{fig:B1xB4}, \ref{fig:B1xB5}.
 
\begin{figure}

\centering
\includegraphics{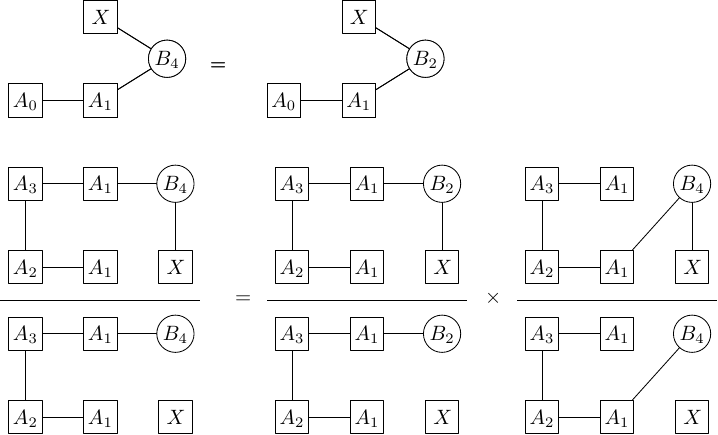}
\caption{The product constraints forcing $B_1\times B_4$ and $D\times B_4$  consist of the depicted constraints, where $X\in\{B_1,D\}$.}
\label{fig:B1xB4}
\end{figure}
 
\begin{figure}

\centering
\includegraphics{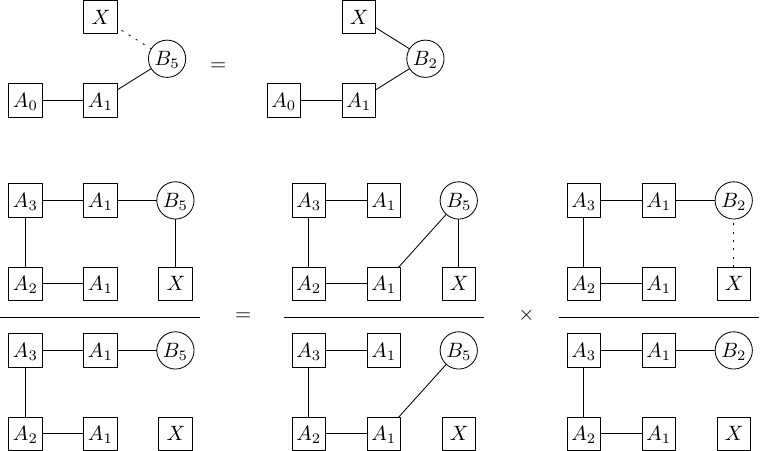}
\caption{The product constraints forcing $B_1\times B_5$ and $D\times B_5$ consist of the depicted constraints, where $X\in\{B_1,D\}$.}
\label{fig:B1xB5}
\end{figure}

\item{\em The projection constraints}
 force the structure of the subgraphon induced by $B_1\times B_1$. They are depicted in Figures~\ref{fig:projection} and~\ref{fig:proj2}.

\begin{figure}
\centering
\includegraphics{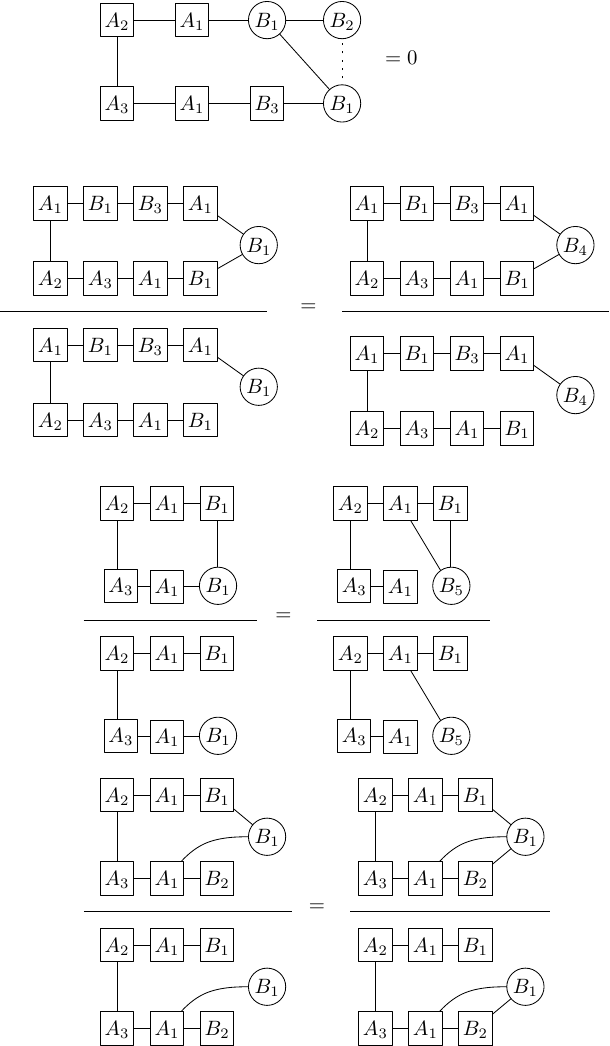}
\caption{The first four projection constraints.}
\label{fig:projection}
\end{figure}

\begin{figure}
\centering
\includegraphics{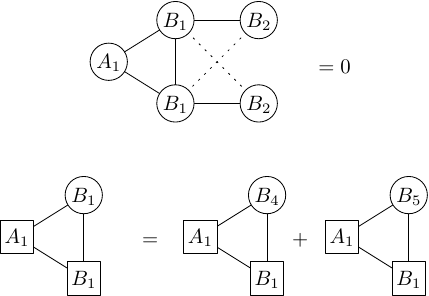}
\caption{The last two projection constraints.}
\label{fig:proj2}
\end{figure}

\item{{\em The infinite constraints}
} force the structure of the subgraphon iduces by $D\times B_1$ and $D\times B_2$. They are depicted in Figure~\ref{fig:infinite}. 
\end{description}

\begin{figure}
\centering
\includegraphics{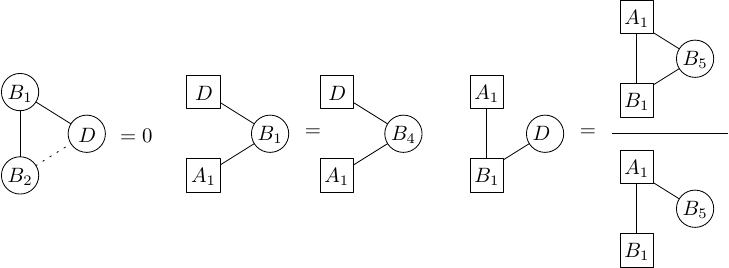}
\caption{The infinite constraints.}
\label{fig:infinite}
\end{figure}

This completes the list of the constraints that are contained in $\constr$

\section{Proof of Theorem \ref{thm:forcible}}
\label{sec:proof}

In this section, we prove Theorem~\ref{thm:forcible}.
In particular, we will show that the \name graphon $\graphon$ is the unique (up to weak isomorphism) graphon satisfying set $\constr$ of the constraints that
we listed in Section~\ref{sec:list}.

Fix a bijective \recipe\ $\ccube=\{\cube_{n}|n\in \nati\}$, which determines the graphon $\graphon$.
Suppose that $W$ is a graphon that satisfies all constraints contained in $\constr$.
Our aim is to show that the graphons $W$ and $\graphon$ are weakly isomorphic.
Since $W$ satisfies the partition constraints,
$W$ is a partitioned graphon with parts of the same measure as those of $\graphon$ and
the vertices in the corresponding parts having the same degree as those in $\graphon$.
The parts of $\WW$ are denoted by $A_0,\ldots, A_3, B_1,\ldots, B_5, C, D, E_1, E_2, F$ in such a way that
the part $X$ corresponds to the part $\X$ of the graphon $\graphon$.
We will strictly use $A_0,\ldots, F$ in the context of the graphon $\WW$ and
$\A_0, \ldots, \F$ in the context of the graphon $\graphon$.
In the analogy to $\B_{1,n}$ and $\B_{2,n}$,
we define $B_{1,n}$ and $B_{2,n}$ to be the vertices of $B_1$ and $B_2$, respectively,
that have relative degree $2^{-n}$ with respect to $A_1$ in $\WW$.

By the Monotone Reordering Theorem (see \cite{bib-lovasz-book} for more details),
there exist measure preserving maps
$\psi_X:X\rightarrow \X$ for $X= A_0, \ldots, A_3$, $ B_2, \ldots, B_5$, $C$, $E_1$, $E_2$, $F$ and
non-decreasing functions $f_{X}:\X\rightarrow [0,1]$
such that $f_X(\psi_{X}(x))=\deg_{C}^{\WW} x$ for almost every $x\in X$. 
Note that we have not (yet) defined the functions $\psi_{B_1}$ and $\psi_D$.

We now define a map $g_n:B_{1,n}\to [0,1]^n$ as
$$g_n(x)=\left(\deg_{B_{2,i}}^{\WW}(x)\right)_{i\in [n]}$$
for $x\in B_{1,n}$ and $g_{\infty}:D\to [0,1]^\infty$ as
$$g_{\infty}(x)=\left(\deg_{B_{2,i}}^{\WW}(x)\right)_{i\in\NN}$$
for $x\in D$.
Note that $g_n$ is well-defined almost everywhere on $B_{1,n}$ and $g_{\infty}$ almost everywhere on $D$.
We next define a map $\psi_{B_1}:B_1\to\B_1$ as
$$\eta_{B_1}^{-1}\left(1-\frac{1}{2^{n-1}}+\frac{\cube_{n}^{-1}(g_n(x))}{2^{n}}\right).$$
for $x\in B_{1,n}$, and
we set $\psi_{B_1}(x)$ to be the same arbitrary vertex of $\B_1$ for $x$ that
does not belong to any $B_{1,n}$, $n\in\NN$.
Similarly,
we define
$$\psi_{D}(x)=\eta_{D}^{-1}(\cube_{\infty}^{-1}(g_{\infty}(x)))\;\mbox{.}$$
Let $\psi$ be the map from $[0,1]$ to $[0,1]$ equal to the map $\psi_X$ on the part $X$
for $X=A_0$, $\ldots$, $A_3$, $B_1$, $\ldots$, $B_5$, $C$, $D$, $E_1$, $E_2$, $F$.

In the rest of the section, we show that the graphons $\graphon^{\psi}$ and $\WW$ are equal almost everywhere and
the map $\psi$ is measure preserving. This would imply that the graphon $\graphon$ is weakly isomorphic to $\WW$.
Note at this point that the maps $\psi_X$ for $X\not=B_1,D$, which form the map $\psi$, are measure preserving;
so we only need to argue that $\psi_{B_1}$ and $\psi_{D}$ are measure preserving maps,
which we will show in Subsections~\ref{sec:proj} and \ref{sec:inf}.

\subsection{Zero and triangular tiles}
\label{subsec:triangle}

The zero constraints guarantee that
if $\graphon$ is equal to zero almost everywhere on $\X\times\Y$ for $X,Y\in\{A_0, \ldots, A_3, B_1, \ldots, B_5,C,D,E_1,E_2,F\}$,
then the graphon $\WW$ is equal to zero almost everywhere on $X\times Y$.
In particular, the graphons $\graphon^{\psi}$ and $\WW$ are equal almost everywhere on $X\times Y$.

The triangular constraints that correspond to those forcing the half-graphon guarantee that
the subgraphon of $\WW$ induced by $C\times C$ is weakly isomorphic to the half-graphon.
The choice of $\psi_C$ now implies that the graphons $\graphon^{\psi}$ and $\WW$ are equal almost everywhere on $C\times C$.
We next analyze the constraints depicted in Figure~\ref{fig:triangular}.
Fix $X\in\{A_0, \ldots, A_3, B_1, \ldots, B_5\}$.
The first constraint in Figure~\ref{fig:triangular} yields that $\deg^\WW_C(z)=\deg^\WW_X(z)$ for almost every $z\in C$.
The second constraint yields that $N_{C}(x\setminus y)$ or $N_{C}(y\setminus x)$ or both
has measure zero for almost every pair $x,y\in X$.
This implies that the graphon $\WW$ has values 0 and 1 almost everywhere on $X\times C$.
The choice of $\psi_X$ implies that $\WW$ and $\graphon^{\psi}$ are equal almost everywhere on $X\times C$
for $X\in \{A_0, \ldots, A_3, B_2, \ldots, B_5\}$. 
Note that we have not reached this conclusion for $X=B_1$ (because $\psi_{B_1}$ is chosen differently)
but we have still shown that the graphon $\WW$ is equal to $0$ or to $1$ almost everywhere on $B_1\times C$ and
that the measure of the set containing $b\in B_1$ such that $\deg_C b\le z$, $z\in [0,1]$, is equal to $z\lambda(B_1)$.

The subgraphon induced by $X\times C$ determines a preorder on the vertices of $X$ according to their relative degrees in $C$.
We often use this fact in our analysis.
In this context,
we write $x\prec_{X} y$ instead of $\deg_{C} x < \deg_{C} y$ for $x,y\in X$.
We also extend this notation to subsets and
write $Y\prec_{X} Z$ for subsets $Y, Z\subseteq X$ if $y\prec_{X}z$ for every $y\in Y$ and every $z\in Z$.

\subsection{Forcing the structure on $A_1\times A_1$}
\label{subsec:A1A1}

We now show that the main \checkered\ constraints,
which are depicted in Figure~\ref{fig:checkered}, force that $\WW$ and $\graphon^{\psi}$ agree almost everywhere on $A_1\times A_1$. 
Our line of arguments follows that in~\cite{bib-rademacher}; we sketch the arguments and refer the reader to~\cite{bib-rademacher} for a more detailed analysis.

The first constraint in Figure~\ref{fig:checkered} implies that
if $x$ is a typical vertex of $A_1$ with respect to $T(\WW)$,
then $\WW(x,y)$ is equal to $0$ or $1$ for almost every $y\in A_1$ and
if $x$ and $x'$ are two typical vertices of $A_1$, then
either $N_{A_1}(x)$ and $N_{A_1}(x')$ are equal up to a set of measure zero or they are disjoint up to a set of measure zero.
Moreover, the measure of the pairs $(x,x')$ such that $W(x,x')\not=1$ and
$N_{A_1}(x)$ and $N_{A_1}(x')$ are equal up to a set of measure zero is zero.
Let $\JJ_{A_1}$ be the set of disjoint non-null measurable subsets of $A_1$ such that
each $J\in\JJ_{A_1}$ is equal to $N_{A_1}(x)$ up to a set of measure zero for some typical vertex $x\in A_1$ and
each $N_{A_1}(x)$ differs from a set contained in $\JJ_{A_1}$ on a set of measure zero.
Our reasoning implies that, except for a subset of $A_1\times A_1$ of measure zero,
$W(x,y)=1$ for $(x,y)\in A_1\times A_1$ if and only if $x$ and $y$ belong to the same set $J\in\JJ_{A_1}$.
Informally speaking, the graphon $\WW$ on $A_1\times A_1$ is a union of disjoint cliques on $J\in\JJ_{A_1}$.
Observe that since the sets contained in $\JJ_{A_1}$ are non-null and disjoint, then $\JJ_{A_1}$ is countable.

The third constraint implies that for every set $J\in\JJ_{A_1}$,
there exists a set $J'\subseteq A_1$ differing from $J$ on a null set such that
$J'$ is an interval with respect to $\prec_{A_1}$,
i.e., if $x,x'\in J'$ and $x\prec_{A_1} x''\prec_{A_1} x'$, then $x''\in J'$.
Hence, we can assume without loss of generality that each $J\in\JJ_{A_1}$ is an interval with respect to $\prec_{A_1}$.
The fourth constraint forces that it holds for almost every two vertices $x\prec_{A_1} x'$ from the same $J\in\JJ_{A_1}$ that
\begin{eqnarray*}
\lambda(J)&=&\lambda\left(\{x''\in A_1\ |\ x\prec_{A_1} x''\mbox{ and }x''\not\in J\}\right)\\
          &=&\lambda\left(\{x''\in A_1\ |\ x''\not\in J\mbox{ and }\exists z\in J\;z\prec_{A_1} x''\}\right)\;\mbox{.}
\end{eqnarray*}	    
Since the second constraint implies that $\sum\limits_{J\in \JJ_{A_1}} \lambda(J)^2=\lambda(A_1)^2/3$,
we obtain (see details of the analysis in~\cite{bib-rademacher}) that
for every $J\in\JJ_{A_1}$, there exists $k\in\NN$ such that $J$ and the set
\[\{x\in A_1\ |\ \deg_C x\in [1-2^{-k-1},1-2^{-k}]\}\]
differ on a set of measure zero.
We conclude that $\WW$ agrees with $\graphon^{\psi}$ almost everywhere on $A_1\times A_1$.

\subsection{Forcing the structure of $A_0\times A_1$}

We now consider the first level constraints, which are depicted in Figure~\ref{fig:forceA0}.
The first constraint implies that $\deg_{A_0}y=0$ or $\deg_{A_1}y=1/2$ for almost every vertex in $y\in A_1$.
In particular, $\WW(x,y)=0$ for almost every $x\in A_0$ and $y\in A_1$ unless $\deg_{A_1}y=1/2$.
The second constraint forces that the density of $\WW$ on $A_0\times A_1$ is equal to $1/2$,
which implies that $\WW(x,y)=1$ for almost every $x\in A_0$ and $y\in A_1$ such that $\deg_{A_1}y=1/2$.
Therefore, $\WW$ is equal to $\graphon^{\psi}$ almost everywhere on $A_0\times A_1$.

\subsection{Forcing remaining \checkered\ subgraphons}

We now use the bipartition constraints, which are depicted in Figure~\ref{fig:bipartite}.
Fix $(X,Y)$ to be one of the pairs $(A_1,A_2)$, $(A_1,B_2)$, $(A_1,B_3)$, $(A_1,B_4)$, $(A_1,B_5)$, $(A_2,A_3)$ and $(A_2,B_2)$.
Note that the list misses the pair $(A_1,B_1)$, which is analyzed separately afterwards.

The first constraint in Figure~\ref{fig:bipartite} implies that
there exist a set $\JJ_X$ formed by disjoint non-null subsets of $X$,
a set $\JJ_Y$ formed by disjoint non-null subsets of $Y$, and
a bijection $f:\JJ_X\to\JJ_Y$ such that
except for a subset of $X\times Y$ of measure zero,
it holds that $W(x,y)=1$ for $(x,y)\in X\times Y$ iff there exist $J\in\JJ_X$ such that $x\in J$ and $y\in f(J)$, and
$W(x,y)=0$ elsewhere on $X\times Y$.
Informally speaking, the graphon $W$ on $X\times Y$ is a disjoint union of complete bipartite subgraphons between $J\in\JJ_X$ and $f(J)\in\JJ_Y$.
We remark here that the set $\JJ_{A_1}$ can in principle differ from the set defined in Subsection~\ref{subsec:A1A1},
however, we will later argue that they actually coincide (in the sense that the elements of the set differ from each other on a set of measure zero).

Analogously to Subsection~\ref{subsec:A1A1},
the second constraint depicted in Figure~\ref{fig:bipartite} implies that
each set contained in $\JJ_X$ differs from an interval with respect to $\preceq_X$ on a set of measure zero and
the third constraint implies that
each set contained in $\JJ_Y$ differs from an interval with respect to $\preceq_Y$ on a set of measure zero.
Hence, we can assume without loss of generality that
each set contained in $\JJ_X$ is an interval with respect to $\preceq_X$ and
each set contained in $\JJ_Y$ is an interval with respect to $\preceq_Y$.
Finally, the fourth constraint implies that the intervals are in the same order, i.e.,
if $J,J'\in\JJ_X$ satisfy that $J\preceq_X J'$, then $f(J)\preceq_Y f(J')$.

It remains to determine the measures of the sets contained in $\JJ_X$ and $\JJ_Y$.
Recall that we have shown that $\WW$ agrees with $\graphon^{\psi}$ almost everywhere on $A_1\times A_1$.
We now split the argument depending on whether $X=A_1$ or $X=A_2$ and start with analyzing the case $X=A_1$.
If $Y\not=A_3$, consider the first constraint depicted in Figure~\ref{fig:aux-check};
this constraint implies that almost all the vertices of $X=A_1$ have the same relative degree with respect to $A_1$ as with respect to $Y$.
Hence, almost every $x\in X$ belongs to some $J\in\JJ_X$ and
for every $J\in\JJ_X$, it holds that $\lambda(f(J))=\lambda(A_1)\deg_{A_1} x$ for almost every $x\in J$.
Since $f$ is a bijection and the sets contained in $\JJ_Y$ are disjoint,
it follows that $\JJ_X$ coincides with the set $\JJ_{A_1}$ defined in Subsection~\ref{subsec:A1A1} and
$\lambda(f(J))=\lambda(J)$ for every $J\in\JJ_X$.
If $Y=A_3$, the third constraint in Figure~\ref{fig:aux-check} yields that
almost all the vertices of $X=A_1$ have either the relative degree with respect to $A_1$ equal to $1/2$ or
the relative degree with respect to $Y$ double the relative degree with respect to $A_1$;
this again implies that $\JJ_X$ coincides with the set $\JJ_{A_1}$ defined in Subsection~\ref{subsec:A1A1} and
$\lambda(f(J))=2\lambda(J)$ for every $J\in\JJ_X$ unless $\lambda(J)=\lambda(X)/2$.
Since the elements of $\JJ_Y$ are disjoint intervals with respect to $\preceq_Y$ and the bijection $f$ preserves their order,
we conclude that for every $J\in\JJ_{Y}$, there exists $k\in\NN$ such that $J$ and the set
\[\{y\in Y\ |\ \deg_C y\in [1-2^{-k-1},1-2^{-k}]\}\]
differ on a set of measure zero (this holds both if $Y=A_3$ and if $Y\not=A_3$).
It follows that $\WW$ and $\graphon^{\psi}$ agree almost everywhere on $X\times Y$ if $X=A_1$,
i.e, $\WW$ and $\graphon^{\psi}$ agree almost everywhere on $A_1\times A_2$, $A_1\times A_3$ and $A_1\times B_2,\ldots,A_1\times B_5$.

We next finish the analysis of the case $X=A_2$; note that $Y$ is either $A_3$ or $B_2$ in this case.
The second constraint depicted in Figure~\ref{fig:aux-check} implies that
almost all the vertices of $X=A_2$ have the same relative degree with respect to $A_1$ as with respect to $Y$.
Hence, almost every $x\in X=A_2$ belongs to some $J\in\JJ_{A_2}$ and
$\lambda(f(J))=\lambda(A_1)\deg_{A_1} x$ for almost every $x\in J$.
Since $f$ is a bijection, we conclude (using the already analyzed structure on $A_2\times A_1$) that
$\lambda(f(J))=\lambda(J)$ for every $J\in\JJ_{A_2}$.
Hence, the set $\JJ_{A_2}$ coincides with the set $\JJ_{A_2}$ as defined in the case $(X,Y)=(A_1,A_2)$, and
for every $J\in\JJ_{Y}$, there exists $k\in\NN$ such that $J$ and the set
\[\{y\in Y\ |\ \deg_{A_1} y\in [1-2^{-k-1},1-2^{-k}]\}\]
differ on a set of measure zero.
It follows that $\WW$ and $\graphon^{\psi}$ agree almost everywhere on $X\times Y$ if $(X,Y)$ is $(A_2,A_3)$ or $(A_2,B_2)$.

In the previous analysis, we have omitted the case $(X,Y)=(A_1,B_1)$.
As in the general case $X=A_1$ considered above,
we derive that the top two constraints in Figure~\ref{fig:bipartite} imply that
there exist a set $\JJ_X$ formed by disjoint non-null subsets of $X$ that are intervals with respect to $\preceq_X$,
a set $\JJ_Y$ formed by disjoint non-null subsets of $Y$, and
a bijection $f:\JJ_X\to\JJ_Y$ such that
except for a subset of $X\times Y$ of measure zero,
it holds that $W(x,y)=1$ for $(x,y)\in X\times Y$ iff there exist $J\in\JJ_X$ such that $x\in J$ and $y\in f(J)$.
Note that we do not make any claims about the structure of the sets contained in $\JJ_Y$.
The first constraint in Figure~\ref{fig:aux-check} implies that $\lambda(f(J))=\lambda(J)$ for every $J\in\JJ_X$.
It follows that $\JJ_X$ coincides with $\JJ_{A_1}$ defined in Subsection~\ref{subsec:A1A1},
almost every vertex $y\in Y$ belongs to $J\in\JJ_Y$ and
the measure of $y\in Y$ with $\deg_{A_1} y=2^{-k}$ is equal to $2^{-k}$.
In particular, the measure of $B_{1,k}$ is $2^{-k}$.
It follows that $\WW$ and $\graphon^{\psi}$ agree almost everywhere on $A_1\times B_1$.

Because each of the sets $\JJ_X$, $X\in\{A_1,\ldots,A_3,B_1,\ldots,B_5\}$, is the same
in all the definitions (in the sense that its elements differ from each other on a set of measure zero) that
we have given in this subsection and Subsection~\ref{subsec:A1A1}, we can just use $\JJ_X$ without referring to the particular place where the set was defined.
We now split each part $X\in\{A_1,\ldots,A_3,B_1,\ldots,B_5\}$ into {\em levels} in the way analogous to that the parts of $\graphon$ are split.
For $k\in\NN$,
the $k$-th level $A_{i,k}$, $i\in\{1,2\}$, of $A_i$ is formed by $x\in A_i$ such that $\deg_{A_1}x=2^{-k}$,
the $k$-th level level $A_{3,k}$ of $A_3$ is formed by $x\in A_3$ such that $\deg_{A_2}x=2^{-k}$, and
the $k$-th level $B_{i,k}$, $i\in\{1,\ldots,5\}$, of $B_i$ is formed by $x\in B_j$ such that $\deg_{A_1}x=2^{-k}$.
The levels of $X$ coincide with sets contained $\JJ_X$ (up to a difference on a set of measure zero).
Note that the measure of the level $A_{i,k}$ or $B_{i,k}$ is $2^{-k}$.
Also note that this coincides with our previous definition of $B_{1,k}$.

\subsection{Using levels in density expressions}

\begin{figure}
\centering
\includegraphics{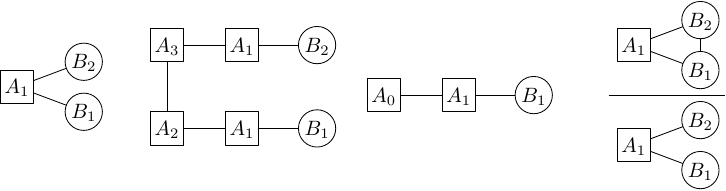}
\caption{Density expressions specifying levels of vertices.}
\label{fig:expr}
\end{figure}

Many of the density expressions used in the following subsections use the level structure of the parts $A_1$ , $A_2$, $A_3$, $B_1 ,\ldots, B_5$ of $\WW$
in combination with the structure of $\WW$ that we have already analyzed.
Some examples of decorated graphs that we use are given in Figure~\ref{fig:expr}.
In the first decorated graph, all three vertices must belong to the same level (ignoring events with probability zero),
i.e., if the root belongs to the $k$-th level of $A_1$, then the expression is equal (with respect to $\WW$) to $2^{-2k}$,
which is the product of the probabilities that a random vertex of $B_1$ belongs to $B_{1,k}$ and that a random vertex of $B_2$ belongs to $B_{2,k}$.

In the second decorated graph,
if the root decorated with $A_2$ belongs to the $k$-th level of $A_2$, which is $A_{2,k}$,
then its neighbors must belong to $A_{1,k}$ and $A_{3,k}$ and the remaining root to $A_{1,k+1}$.
In such case, the expression is equal to $2^{-2k-1}$,
which is the product of the probabilities that a random vertex of $B_1$ belongs to $B_{1,k}$ and that a random vertex of $B_2$ belongs to $B_{2,k+1}$.
In the third decorated graph,
the root decorated with $A_1$ must belong to $A_{1,1}$ and
the expression is equal to $1/2$,
which is the probability that a random vertex of $B_1$ belongs to $B_{1,1}$.

The final expression is more complex. Suppose that the root belongs to $A_{1,k}$.
The denominator is equal to $2^{-2k}$ as we have discussed earlier.
The numerator is equal to $2^{-2k}$ multiplied by the density between $B_{1,k}$ and $B_{2,k}$,
i.e., the whole expression is equal to the density of $\WW$ between the $B_{1,k}$ and $B_{2,k}$.

\subsection{Stair constraints} 

We now focus on the stair constraints, which are depicted in Figure~\ref{fig:stairs}.
They are intended to force the desired structure on $B_1\times B_3$.
The first constraint in Figure~\ref{fig:stairs} determines the relative degrees of vertices of $B_1$ in $B_3$,
i.e., it enforces that $\deg_{B_3}x=1-2^{-k}$ for almost every $x\in B_{1,k}$.
The second constraint forces that the following holds for almost every vertex $x\in B_1$ and every $k\in\NN$:
if $\deg_{B_{3,k+1}}x>0$, then $\deg_{B_{3,k}}x=1$ (if $\deg_{B_{3,k}}x=0$,
then there exists a choice of the roots decorated with $A_1$, $A_2$, $A_3$ and $A_1$ such that
the density expression is non-zero with $x$ being the root decorated with $B_1$).
Consequently, for almost every $x\in B_1$, there exists $k_0\in\NN$ such that
$\WW(x,y)=1$ for almost every $y\in B_{3,k}$, $k<k_0$ and $\WW(x,y)=0$ for almost every $y\in B_{3,k}$, $k>k_0$.
However, it is possible that $\deg_{B_3}x=1-2^{-k}$ for almost every $x\in B_{1,k}$ only if
it holds that for almost every $x\in B_1$, $k_0$ is equal to the level of $x$ and $\WW(x,y)=1$ for almost every $y\in B_{3,k_0}$.
It follows that $\WW$ agrees with $\graphon^{\psi}$ almost everywhere on $B_1\times B_3$.

\subsection{Coordinate constraints}
\label{sec:coord}

The coordinate constraints from Figure~\ref{fig:coord} force basic structure
between the parts $B_1$ and $D$ on one side and the parts $B_2$, $B_4$ and $B_5$ on the other side.
Fix $Y$ to be one of the parts $B_2$, $B_4$ and $B_5$.
The first constraint depicted in Figure~\ref{fig:coord} implies that
almost every vertex $b$ of $B_1$ with non-zero relative degree with respect to $Y_k$
has relative degree one with respect to $B_{3,k}$.
Hence, almost every $b\in B_{1,k}$ satisfies that $W(b,y)>0$ only if $y$ belongs to $Y_{k'}$ with $k'\le k$ except for a set of measure zero;
in particular, $W(b,y)=0$ for almost every $b\in B_{1,k}$ and $y\in Y_{k'}$, $k'>k$.

In addition to $Y$, fix $X$ to be either $B_1$ or $D$.
The second constraint implies that $N_{X}(y\setminus y')$ has measure zero for every $k$ and almost every two $y,y'\in Y_k$ such that $y\prec_{Y} y'$; consequently,
$\WW$ is equal to $0$ or $1$ almost everywhere on $X\times Y$.
It follows that for almost every $x\in X$ and every $k$,
there exists $y_0\in Y_k$ such that $\WW(x,y)=1$ for almost every $y\in Y_k$ with $y\preceq_{Y} y_0$ and
$\WW(x,y)=0$ for almost every $y\in Y_k$ with $y_0\preceq_{Y} y$.
In particular, the definition of $\psi$ on $B_1$ and $D$ now yields that
$\WW=\graphon^{\psi}$ almost everywhere on $B_1\times B_2$ and $D\times B_2$. 

\subsection{Initial coordinate constraint}

We now consider the initial coordinate constraint, which can be found in Figure~\ref{fig:first}.
The decorated graphs appearing in the constraint are evaluated to the following quantities when $b$ is the root decorated with $B_1$:
$$\deg_{B_{2,1}}b=\frac{\deg_{C}b-(1-2\deg_{A_1}b)}{\deg_{A_1}b}\;\mbox{.}$$
Consider now $b\in B_{1,k}$. Unless $b$ belongs to an exceptional set of measure zero,
the right hand side belongs to the interval $[0,1]$ only if $\deg_{C}b$ belongs to the interval $[1-2^{-k+1},1-2^{-k}]$.
This implies that $\WW$ agrees with $\graphon^{\psi}$ almost everywhere on $B_1\times C$.

The results of Subsection~\ref{subsec:triangle} imply that the measure of $b\in B_1$ with $\deg_{C}b\le z$ for $z\in [0,1]$ is equal to $z\lambda(B_1)$.
Hence, it follows for every $z\in [0,1]$ and every $k\in\NN$ that
\begin{equation}
\lambda\left(\{b\in B_{1,k}\ |\ (g_k(b))_1\le z\}\right)=z\cdot\lambda\left(B_{1,k}\right)\;\mbox{.}
\label{eq2.1}
\end{equation}

\subsection{Distribution constraints}
\label{sec:dist}

The equality \eqref{eq2.1} can be interpreted as saying that the first coordinate of each $g_k$ is uniformly distributed.
We now argue that the same holds for the remaining coordinates of $g_k$, $k\in\NN$, and all the coordinates of $g_{\infty}$.

The decorated graphs appearing in the first constraint in Figure~\ref{fig:distr} are evaluated to the following quantities
for every $k\in\NN$ and almost every $b\in B_{2,k'}$, $k'\le k$:
$$\deg_{B_{1,k}}b = 1-\frac{\deg_{C}b-(1-2\deg_{A_{1}}b)}{\deg_{A_{1}}b}\;\mbox{.}$$
Almost every $b\in B_{2,k'}$ satisfies that $\deg_{A_{1}}b=2^{-k'}$; so, we get that
$$\deg_{B_{1,k}}b = 1-2^{k'}\left(\deg_{C}b-\left(1-2^{-(k'-1)}\right)\right)\;\mbox{.}$$
Informally speaking,
the relative degree of almost every $b\in B_{2,k'}$ decreases linearly from $1$ to $0$ with its position within $B_{2,k'}$ given by $\prec_{B_2}$.
This and the analysis of the structure between the parts $B_1$ and $B_2$ in Subsection~\ref{sec:coord} imply that
\begin{equation}
\lambda\left(\{b\in B_{1,k}\ |\ (g_k(b))_{k'}\le z\}\right)=z\cdot\lambda\left(B_{1,k}\right)
\label{eq2.2}
\end{equation}
for every $k\in\NN$, every $k'\in [k]$ and every $z\in [0,1]$.

The second constraint depicted in Figure~\ref{fig:distr} implies the analogous statement for the structure between $D$ and $B_2$.
In particular, it holds that
\begin{equation}
\lambda\left(\{d\in D\ |\ (g_{\infty}(d))_{k'}\le z\}\right)=z\cdot\lambda(D)
\label{eqinf}
\end{equation}
for every $k'\in\NN$ and $z\in [0,1]$.

\subsection{Product constraints}
\label{sec:prod}

We now analyze the product constraints, which are depicted in Figures~\ref{fig:B1xB4} and~\ref{fig:B1xB5}.
Fix $(X,Y)$ to be one of the pairs $(B_1,B_4)$, $(B_1,B_5)$, $(D,B_4)$ and $(D,B_5)$.
The results on the structure of the graphon $\WW$ between $X$ and $Y$ from Subsection~\ref{sec:coord} imply that
if we show that $\deg_{Y_{k'}}x=\deg_{\Y_{k'}}\psi(x)$ for every $x\in X$ and $k'\in\NN$,
then $\WW$ and $\graphon^{\psi}$ agree almost everywhere on $X\times Y$.

Suppose that $(X,Y)=(B_1,B_4)$.
The first constraint depicted in Figure~\ref{fig:B1xB4} implies that $\deg_{B_{4,1}} b= \deg_{B_{2,1}} b$,
i.e., $\deg_{B_{4,1}} b=(g_k(b))_{1}$, for almost every $b\in B_{1,k}$.
The second constraint yields that $\deg_{B_{4,i}} b=\deg_{B_{2,i}}b \cdot \deg_{B_{4,i-1}} b$ for almost every $b\in B_{1,k}$;
if $i>k$, then $\deg_{B_{2,i}}b$ is equal to zero and so is $\deg_{B_{4,i}} b$ for such $b$.
If $i\le k$, we obtain that it holds for almost every $b\in B_{1,k}$ that
$$\deg_{B_{4,i}} b=(g_k(b))_i\cdot\deg_{B_{4,i-1}} b=\prod_{i'=1}^i(g_k(b))_{i'}\;\mbox{.}$$
Hence, the graphons $\WW$ and $\graphon^{\psi}$ are equal almost everywhere on $X\times Y=B_1\times B_4$.
The remaining three choices of $X\times Y$ are analyzed in the completely analogous way.

\subsection{Projection constraints}
\label{sec:proj}

This subsection forms the core of our argument.
We show that the mapping $\psi_{B_1}$ is measure preserving;
this will be implied by proving the following identity for every $k\in\NN$.
\begin{equation}
\lambda\left(\{b\in B_{1,k}\ |\ (g_k(b))_i\leq z_i\ \forall i\in [k]\}\right)=\lambda(B_{1,k}) \prod_{i\in [k]}z_i\;\mbox{ for all }z\in [0,1]^{k}. 
\label{eq3}
\end{equation}
Note that if \eqref{eq3} holds, then $g_k(B_{1,k}\setminus Z)$,
i.e., the image of $g_k$ in $[0,1]^k$ after removing $Z$ from the domain, is dense for every $Z\subseteq B_{1,k}$ of measure zero.

We prove \eqref{eq3} by induction on $k\in\NN$. Note that \eqref{eq3} holds for $k=1$ by \eqref{eq2.1}.
As a part of the induction argument,
we will also show that $\WW$ and $\graphon^{\psi}$ are equal almost everywhere on $B_{1,k'}\times B_{1,k}$ if $k'<k$.

Fix integers $k'$ and $k$ such that $k'<k$ and assume using the induction that (\ref{eq3}) holds for all smaller values of $k$.
The first constraint depicted in Figure~\ref{fig:projection} yields that
$N_{B_2}(b'\setminus b)$ has measure zero, i.e., $\deg_{B_{2,i}}b'\leq \deg_{B_{2,i}}b$,
for almost every pair of vertices $b'\in B_{1,k'}$ and $b\in B_{1,k}$ such that $W(b',b)>0$.
In other words, the set
$$N_{B_{1,k'}}(b) \setminus \{b'\in B_{1,k'}\ |\ (g_{k'}(b'))_i\leq (g_{k}(b))_i\ \forall i\in[k']\}$$
has measure zero for almost every $b\in B_{1,k}$ and the set
$$N_{B_{1,k}}(b') \setminus \{b\in B_{1,k}\ |\ (g_{k}(b))_i\geq (g_{k'}(b'))_i\ \forall i\in[k']\}$$
has measure zero for almost every $b'\in B_{1,k'}$.

The second constraint in Figure~\ref{fig:projection} forces that
$$\deg_{B_{1,k'}}b=\deg_{B_{4,k'}}b$$
for almost every $b\in B_{1,k}$.
Recall that we have shown in Subsection~\ref{sec:prod} that 
\begin{equation}\label{eqx}
\deg_{B_{4,k'}}b=\prod_{i\in[k']}(g_{k}(b))_i
\end{equation}
for almost every $b\in B_{1,k}$.
Since the equality (\ref{eq3}) holds for $k'$, it follows that the set $N_{B_{1,k'}}(b)$ and the set 
\[\{b'\in B_{1,k'}\ |\ (g_{k'}(b'))_i\leq (g_{k}(b))_i\ \forall i\in [k']\}\]
differ on a set of measure zero and
therefore the graphon $\WW$ is equal to $1$ almost everywhere on $N_{B_{1,k'}}(b)$ for almost every $b\in B_{1,k}$.
Hence, $\WW$ and $\graphon^{\psi}$ agree almost everywhere on $B_{1,k'}\times B_{1,k}$.
It follows that $N_{B_{1,k}}(b')$ and the set
\begin{equation}
\{b\in B_{1,k}\ |\ (g_{k}(b))_i\geq (g_{k'}(b'))_i\ \forall i\in[k']\}
\label{eqy}
\end{equation}
differ on a set of measure zero.

We now present the induction step for proving \eqref{eq3} by showing that it holds for $k$ assuming that
\eqref{eq3} holds for the previous value of $k$, i.e., $k-1$.
The third constraint depicted in Figure~\ref{fig:projection} guarantees that
\[\deg_{B_{1,k}}b'=\deg_{B_{5,k-1}}b'\;\mbox{,}\]
which yields that
\begin{equation*}
\deg_{B_{1,k}}b'=\prod_{i\in [k-1]} (1-(g_{k-1}(b'))_i)
\end{equation*}
for almost every $b'\in B_{1,k-1}$.
Since $N_{B_{1,k}}(b')$ and the set (\ref{eqy}) differ on a set of measure zero,
we get that
\begin{equation}
\frac{\lambda(\{b\in B_{1,k}\ |\ (g_{k}(b))_i\geq (g_{k-1}(b'))_i\ \forall i\in [k-1]\})}{\lambda(B_{1,k})}=\prod_{i\in [k-1]} (1-(g_{k-1}(b')_i))\label{eq4}
\end{equation}
for almost every $b'\in B_{1,k-1}$.

The fourth constraint in Figure~\ref{fig:projection} implies that 
\begin{equation}
\deg_{B_{1,k}}b'=\deg_{N_{B_{1,k}}(x)}b'
\label{eqz}
\end{equation}
for almost every $b'\in B_{1,k-1}$ and almost every $x\in B_{2,k}$ (the vertex $b'$ is the root labeled with $B_1$ and
the vertex $x$ is the root labeled with $B_2$ in the constraint);
note that $W(x,y)=1$ almost every $x\in B_{2,k}$ and for almost every $y\in N_{B_{1,k}}(x)$,
by the structure of the graphon established in Subsection~\ref{sec:coord}.
The structure of the graphon established in Subsection~\ref{sec:coord} also implies the following:
it holds for almost every $x\in B_{2,k}$ that
the set $N_{B_{1,k}}(x)$ is the set of vertices $y\in B_{1,k}$ with $(g_k(y))_k\ge \zeta$ for some $\zeta\in [0,1]$, and
it holds for almost every $\zeta\in [0,1]$ that there exists $x\in B_{2,k}$ such that
the set $N_{B_{1,k}}(x)$ is the set of vertices $y\in B_{1,k}$ with $(g_k(y))_k\ge \zeta$.
Hence, the equality \eqref{eqz} guarantees that
\begin{eqnarray}
& & \frac{\lambda(\{b\in B_{1,k}\ |\ (g_{k}(b))_i\geq (g_{k-1}(b'))_i\ \forall i\in[k-1]\})}{\lambda(B_{1,k})}\label{eq5}\\
&=&\frac{\lambda(\{b\in B_{1,k}\ |\ (g_{k}(b))_i\geq(g_{k-1}(b'))_i\ \forall i\in[k-1]\mbox{ and } (g_{k}(b))_{k}\geq \zeta\})}{\lambda(\{b\in B_{1,k}|(g_{k}(b))_k\geq \zeta\})}\nonumber
\end{eqnarray}
for almost every $b'\in B_{1,k-1}$ and almost every $\zeta\in [0,1]$.
The equality \eqref{eq2.2} implies that
$$\lambda(\{b\in B_{1,k}|(g_{k}(b))_k\geq \zeta\})=(1-\zeta)\lambda(B_{1,k})$$
for every $\zeta\in [0,1]$.
This combined with (\ref{eq4}) and (\ref{eq5}) yields that
$$\frac{\lambda(\{b\in B_{1,k}\ |\ (g_{k}(b))_i\geq(g_{k-1}(b'))_i\ \forall i\in[k-1]\mbox{ and } (g_{k}(b))_{k}\geq \zeta\})}{(1-\zeta)\lambda(B_{1,k})}$$
$$=\prod\limits_{i\in [k-1]} (1-(g_{k-1}(b)_i))$$
for almost every $b'\in B_{1,k-1}$ and almost every $\zeta\in [0,1]$.
Since the image of $g_{k-1}$ is dense even after removing a set of measure zero from its domain,
we conclude that
$$\frac{\lambda(\{b\in B_{1,k}\ |\ (g_{k}(b))_i\geq z_i\ \forall i\in[k]\})}{\lambda(B_{1,k})}=\prod_{i\in[k]} (1-z_{i})$$
for every $z\in [0,1]^k$. 
However, this is equivalent (by applying a straightforward manipulation using the principle of inclusion and exclusion) to (\ref{eq3}) for $k$.
Since $g_k$ satisfies \eqref{eq3}, the map $\psi_{B_{1,k}}$ is measure preserving.
Consequently, $\psi_{B_1}$ is a measure preserving map.

We have shown that the graphons $\WW$ and $\graphon^{\psi}$ agree almost everywhere on $B_{1,k'}\times B_{1,k}$ for $k'\not=k$.
It remains to analyze the structure of the graphon $\WW$ on $B_{1,k}\times B_{1,k}$, $k\in\NN$.
Fix $k\in\NN$.
The first constraint in Figure~\ref{fig:proj2} forces that
$N_{B_2}(b'\setminus b)$ or $N_{B_2}(b\setminus b')$ has measure zero for almost all $b, b' \in B_{1,k}$ with $\WW(b,b')>0$.
Hence, $N_{B_{1,k}}(b)$ is a subset of the set
\begin{equation}
  \{b'\in B_{1,k}\ |\ (g_k(b))_i\le (g_k(b'))_i\ \forall i\in [k]\}\cup
  \{b'\in B_{1,k}\ |\ (g_k(b))_i\ge (g_k(b'))_i\ \forall i\in [k]\} 
\label{eq-proj2}  
\end{equation}
for almost every $b\in B_{1,k}$.
Since $\psi_{B_1}$ is a measure preserving map and
the graphons $\WW$ and $\graphon^{\psi}$ are equal almost everywhere on $B_1\times B_2$, $B_1\times B_4$ and $B_1\times B_5$,
it follows that the measure of the set given in \eqref{eq-proj2} is equal to
$$\lambda(B_{1,k})\left(\prod_{i=1}^{k} \deg_{B_{2,i}}b +\prod_{i=1}^{k} (1-\deg_{B_{2,i}}b)\right)$$
$$=\lambda(B_{1,k})\left(\deg_{B_{4,k}} b + \deg_{B_{5,k}}b\right)$$
for almost every $b\in B_{1,k}$.

The second constraint in Figure~\ref{fig:proj2} implies that $\deg_{B_{1,k}} b=\deg_{B_{4,k}} b + \deg_{B_{5,k}}b$ for almost every $b\in B_{1,k}$.
Hence, it must hold that $N_{B_{1,k}}(b)$ is the set given in \eqref{eq-proj2} and
$\WW(b,b')=1$ for almost every $b\in B_{1,k}$ and $b'\in N_{B_{1,k}}(b)$.
We conclude that the graphons $\WW$ and $\graphon^{\psi}$ agree almost everywhere on $B_{1,k}\times B_{1,k}$ for every $k\in\NN$.

\subsection{Infinite constraints}
\label{sec:inf}

In this subsection,
we establish that the graphon $\WW$ is equal to $\graphon^{\psi}$ almost everywhere on $B_{1}\times D$ by proving that
the two graphons are equal almost everywhere on $B_{1,k}\times D$ for every $k\in\NN$.
We also establish that $\psi_{D}$ is a measure preserving map by showing for every $k\in\NN$ that
\begin{equation}
\lambda\left(\{d\in D\ |\ (g_{\infty}(d))_i\leq z_i\ \forall i\in [k]\}\right)=\lambda(D) \prod_{i\in [k]}z_i\;\mbox{ for all }z\in [0,1]^{k}
\label{eq3i}
\end{equation}
Fix $k\in\NN$ for the rest of the subsection.

Let $d\in D$ and $b\in B_{1,k}$. We define 
$$M^k_{B_1}(d)=\{b\in B_{1,k}\ |\ \forall\ i\in[k] \deg_{B_{2,i}}b\leq\deg_{B_{2,i}}d\}\mbox{ and}$$
$$M^k_{D}(b)=\{d\in D\ |\ \forall\ i\in[k] \deg_{B_{2,i}}d\geq\deg_{B_{2,i}}b\}.$$
We obtain using~(\ref{eq3}) that
$$\lambda(M^k_{B_1}(d))=\lambda(B_{1,k})\cdot \prod_{i=1}^{k}\deg_{B_{2,i}}d=\lambda(B_{1,k})\cdot \prod_{i=1}^{k} (g_{\infty}(d))_i$$
for almost every $d\in D$.

We now analyze the constraints depicted in Figure~\ref{fig:infinite}.
The first constraint forces that $N_{B_2}(b\setminus d)$ has measure zero for almost every $b\in B_1$ and almost every $d\in D$ with $\WW(b,d)>0$. 
It follows that
the set $N_{B_{1,k}}(d)$ is a subset of $M^k_{B_1}(d)$ up to a set of measure zero for almost every $d\in D$ and
the set $N_{D}(b)$ is a subset of $M^k_{D}(b)$ up to a set of measure zero for almost every $b\in B_{1,k}$.

The second constraint in Figure~\ref{fig:infinite} implies that $\deg_{B_{1,k}} d = \deg_{B_{4,k}} d $ for almost every $d\in D$.
We have shown in Subsection~\ref{sec:prod} that $\deg_{B_{4,k}} d=\prod_{i=1}^{k}\deg_{B_{2,i}} d$ for almost every $d\in D$.
Therefore,
\[
\lambda\left(N_{B_{1,k}}(d)\right) \geq
\lambda(B_{1,k})\cdot \deg_{B_{1,k}} d = \lambda(B_{1,k}) \prod_{i=1}^{k}\deg_{B_{2,i}} d
=\lambda\left(M^k_{B_1}(d)\right)
\]
for almost every $d\in D$.
It follows that the sets $N_{B_{1,k}}(d)$ and $M^k_{B_1}(d)$ differ on a set of measure zero and
$W(d,b)=1$ for almost every $d\in D$ and almost every $b\in N_{B_{1,k}}(d)$.
This determines the structure of $\WW$ on $B_{1}\times D$.
In particular, it follows that $\WW$ and $\graphon^{\psi}$ are equal almost everywhere on $B_1\times D$.
Also note that the sets $N_{D}(b)$ and $M^k_{D}(b)$ differ on a set of measure zero for almost every $b\in B_{1,k}$.

We now show that $g_{\infty}$ satisfies \eqref{eq3i}.
The third constraint in Figure~\ref{fig:infinite} implies that 
$$\deg_{D} b=\deg_{B_{5,k}} b=\prod_{i=1}^{k} (1-\deg_{B_{2,i}} b)$$
for almost every $b\in B_{1,k}$. 
Since $\deg_{D} b=\lambda(N_D(b))/\lambda(D)$ for almost every $b\in B_{1,k}$,
we deduce that
\begin{equation}
\frac{\lambda(\{d\in D\ |\ \forall\ i\in[k]\ (g_{\infty}(d))_i\geq\deg_{B_{2,i}}b\})}{\lambda(D)}=\prod_{i=1}^{k}(1-\deg_{B_{2,i}} b)=\prod_{i=1}^{k}(1-(g_k(b))_i)
\label{eqlast}
\end{equation}
for almost every $b\in B_{1,k}$.
Since the image of $g_k$ is dense in $[0,1]^k$ even after removing a set of measure zero from its domain,
we conclude that $g_{\infty}$ satisfies \eqref{eq3i}.
It follows that $\psi_D$ is measure preserving.

\subsection{Structure involving the parts $E_1, E_2$ and $F$}

Let $I=[0,1]\setminus(E_1\cup E_2 \cup F)$.
The degree unifying constraints, which are depicted in Figure~\ref{fig:XxE},
imply that for every $X\in \{A_0, \ldots, A_3, B_1,\ldots, B_5, C\}$ and almost every $x,x'\in X$:
$$\frac{1}{\lambda(E_1)}\int\limits_{E_1}\WW(x,z)\,\dif z=(1-\deg_{I}x) \mbox{ and}$$
$$\frac{1}{\lambda(E_1)}\int\limits_{E_1}\WW(x,z)\WW(x',z)\,\dif z=(1-\deg_{I}x)(1-\deg_{I}x').$$
The reasoning given in~\cite[proof of Lemma 3.3]{bib-lovasz11+} implies that the latter identity
holds for almost every $x=x'\in X$, i.e., it holds that
$$\frac{1}{\lambda(E_1)}\int\limits_{E_1}(\WW(x,z))^2\,\dif z=(1-\deg_{I}x)^2$$
for almost every $x\in X$.
The Cauchy-Schwarz inequality yields that $\WW(x,z)=1-\deg_{I}x$ for almost every $x\in X$ and $z\in E_1$.
This implies that $\deg_{[0,1]\setminus (E_2\cup F)} x=1/2$ for almost every $x\in I\setminus D$.
Since the graphons $\graphon^{\psi}$ and $\WW$ agree almost everywhere on $I^2$,
almost every $x\in I$ must have the same relative degree on $I$ in both $\graphon^{\psi}$ and $\WW$.
It follows that $\graphon^{\psi}$ and $\WW$ agree almost everywhere on $I\times E_1$.

Similarly, the constraints depicted in  Figure~\ref{fig:XxE2} imply that 
$$\deg_{B_1\cup B_2\cup B_4\cup B_5 \cup E_2} x=1/2$$
for almost every $x\in D$, which implies that $\graphon^{\psi}$ and $\WW$ are equal almost everywhere on $D\times E_2$.
Finally, the two degree distinguishing constraints yield that the graphon $\WW$ on $X\times F$, for $X=A_1,\ldots$, $A_3$, $B_1,\ldots, B_5$, $C$,
is constant and its density is the one given by Table~\ref{tab:FxX}.
We conclude that the graphons $\graphon^{\psi}$ and $\WW$ are equal almost everywhere.

\section{Conclusion}
\label{sec:concl}

The method for establishing that a graphon is finitely forcible using decorated constraints,
which originated in~\cite{bib-rademacher} and was further developed in this paper,
turned out to be useful in several follow up results, which we now mention.
First, Cooper et al.~\cite{bib-fup-CKKN} addressed one of the motivations for Conjecture~\ref{conj:2} and
constructed a finitely forcible graphon $W$ such that the number of parts in every weak $\varepsilon$-regular partition of $W$
is at least $2^{\Omega\left(\varepsilon^{-2}/2^{5\log^*\varepsilon^{-2}}\right)}$
for an infinite sequence of $\varepsilon$ tending to $0$.
This almost matches the general upper bound of $2^{\Theta(\log^2\varepsilon^{-1})}$
on the number of parts in weak $\varepsilon$-regular partitions~\cite{bib-frieze99+}.
It is worth noting that
while for any $\varepsilon>0$, there is a graphon $W$ such that
each weak $\varepsilon$-regular $W$ has at least $2^{\Omega(\varepsilon^{-2})}$ parts,
there is no graphon such that
every weak $\varepsilon$-regular partition of $W$ is at least $2^{\Omega(\varepsilon^{-2})}$ parts
for an infinite sequence of $\varepsilon$ tending to zero.
The line of research on constructions of complex finitely forcible graph limits
culminated with the result of Cooper et al.~\cite{bib-fup-CKM} that
every graphon is a subgraphon of a finitely forcible graphon.
This general result of Cooper~et~al.~was also a key ingredient in the argument of Grzesik~et~al.~in~\cite{bib-fup-GKL}
for disproving a conjecture of Lov\'asz that every extremal graph theory problem has a finitely forcible optimum,
which was one of the most cited open problems on dense graph limits.

\section*{Acknowledgments}

The authors would like to thank Jan Volec for his comments on the results contained in this paper and very useful suggestions regarding their presentation.
They would also like to thank L\'aszl\'o Lov\'asz for his comments on the relation of the dimension of the space of typical vertices to other aspects of graph limits.
Last but not least, the authors would like to thank all the anonymous referees for their insightful comments that
have helped to improve the presentation of the paper very significantly.

\end{document}